%% file: Algebraic_two-level_measure_trees_7.tex
\documentclass[11pt]{amsart}

\usepackage{latexsym, amssymb, amsthm,color}
\usepackage[centertags]{amsmath}
\usepackage{pictexwd}
\usepackage[normalem]{ulem}
\usepackage[all]{xy}
\usepackage{bm}

\usepackage{pictex}
\usepackage{pstricks}
\usepackage{multido}
\usepackage{xfp}

\input ./figures.tex

\definecolor{myblue}{rgb}{0,0,.8}
\definecolor{myred}{rgb}{.8,0,0}
\usepackage[bookmarks=true,breaklinks=true, urlcolor=myred, hyperindex,dvips,ps2pdf,citecolor=myred,colorlinks,linkcolor=myblue,pagebackref]{hyperref}

\definecolor{GREEN}{rgb}{0,1,0}
\definecolor{green4}{rgb}{.1,.5,.1}
\definecolor{blue}{rgb}{0,0,1}
\definecolor{gray}{rgb}{0.5,0.5,0.5}

  \newtheorem{theorem}{Theorem}
  \newtheorem{definition}{Definition}[section]

  \newtheorem{cond}[definition]{Condition}
  \newtheorem{proposition}[definition]{Proposition}
  \newtheorem{lemma}[definition]{Lemma}

  \newtheorem{corollary}[definition]{Corollary}
  \newcommand{\beCond}[2]{\Rand{\vspace{0,6cm}\tt #1}\begin{cond}[#2]
  \label{#1}} 
  \theoremstyle{definition}
  \newtheorem{remark}[definition]{Remark}
  \newtheorem{example}[definition]{Example}
  \numberwithin{equation}{section}
  \newtheoremstyle{step}{3pt}{0pt}{\itshape}{}{\bf}{}{.5em}{}

\theoremstyle{step} \newtheorem{step}{Step}

\newcommand{\D}{\mathbb{D}}

\newcommand{\M}{\mathbb{M}}
\newcommand{\PP}{\mathbb{P}}
\newcommand{\R}{\mathbb{R}}

\newcommand{\N}{\mathbb{N}}

\newcommand{\BS}{\mathbb{S}}
\newcommand{\T}{\mathbb{T}}

\newcommand{\CB}{\mathcal{B}}
\newcommand{\CC}{\mathcal{C}}

\newcommand{\CE}{\mathcal{E}}
\newcommand{\CF}{\mathcal{F}}

\newcommand{\CI}{\mathcal{I}}

\newcommand{\CM}{\mathcal{M}}
\newcommand{\CP}{\mathcal{P}}

\newcommand{\CS}{\mathcal{S}}
\newcommand{\CT}{\mathcal{T}}
\newcommand{\KK}{\mathcal{K}}

\newcommand{\FC}{\mathfrak{C}}
\newcommand{\Fs}{\mathfrak{s}}

\newcommand{\Fp}{\mathfrak{p}}

\newcommand{\vn}{\underline{n}}
\newcommand{\vN}{\underline{N}}
\newcommand{\mat}[1]{\underline{\underline{#1}}}
\newcommand{\du}[1]{\mathrm{d}\underline{#1}}

\newcommand{\mau}{\underline{\underline{u}}}

\newcommand{\supp}{\mathrm{supp}}
\newcommand{\comp}{\mathrm{comp}}
\newcommand{\seg}{\mathrm{seg}}
\newcommand{\br}{\mathrm{br}}
\newcommand{\lf}{\mathrm{lf}}
\newcommand{\at}{\mathrm{at}}
\newcommand{\edge}{\mathrm{edge}}
\newcommand{\conv}{\mathrm{conv}}

\newcommand{\josue}[1]{{\color{black}#1}}

\DeclareMathAlphabet{\mathpzc}{OT1}{pzc}{m}{it}

\begin{document}

\title[Algebraic two-level measure trees]{Algebraic two-level measure trees}

\author[J. Nussbaumer, V.C. Tran and A. Winter]{Josu\'e Nussbaumer \and Viet Chi Tran \and Anita Winter}

\address{LAMA, Univ Gustave Eiffel, Univ Paris Est Creteil, CNRS, F-77454 Marne-la-Vall\'ee, France. %\href{josue.nussbaumer@gmail.com}
}
%\address{LAMA, Univ Gustave Eiffel, Univ Paris Est Creteil, CNRS, F-77454 Marne-la-Vall\'ee, France; \href{chi.tran@univ-eiffel.fr}
%}
\address{Fakult\"at f\"ur Mathematik, Universit\"at Duisburg-Essen, Germany. %; \href{anita.winter@uni-due.de}
}

% \email{josue.nussbaumer@u-pem.fr}

\thispagestyle{empty}
\date{\today}



\begin{abstract}
With the algebraic trees, L\"ohr and Winter (2021) introduced a generalization of the notion of graph-theoretic trees to account for potentially uncountable structures. The tree structure is given by the map which assigns to each triple of points their branch point. No edge length or distance is considered. One can equip a tree with a natural topology and a probability measure on the Borel-$\sigma$-field, defining in this way an algebraic measure tree. The main result of L\"ohr and Winter is to provide with the sample shape convergence a compact topology on the space of binary algebraic measure trees. This was proved by encoding the latter with triangulations of the circle. In the present paper, we extend this result to a two level setup. Motivated by the study of hierarchical systems with two levels in biology, such as host-parasite populations, we equip algebraic trees with a probability measure on the set of probability measures. To show the compactness of the space of binary algebraic two-level measure trees, we enrich the encoding of these trees by triangulations of the circle, by adding a two-level measure on the circle line. As an application, we define the two-level algebraic Kingman tree, that is the random algebraic two-level measure tree obtained from the nested Kingman coalescent.
\end{abstract}

 \maketitle

% {%\small
% \tableofcontents
% }

\section{Introduction}

Because of their applications to biology and computer science, trees have received significant interest in the mathematical literature of the last decades. In probability theory in particular, many random tree structures have been introduced to model genealogical (or phylogenetic) trees and their evolutions over time. The simplest of these models are defined on state spaces of trees with a finite number of vertices. The size of the space of all trees with a given number of vertices grows exponentially, and it becomes hard to study qualitative statistics of the structures. To overcome this issue, it seems reasonable to consider continuum limits of tree models and study their properties. We are thus interested in a setup which unifies discrete and continuum trees, and where the focus is put on the structure of the trees. Following L\"ohr and Winter \cite{LoehrWinter2021}, we will consider here algebraic trees with a particular attention to two-level (possibly uncountable) trees. The latter, and more generally trees with nested or hierarchical structures, can serve for the modelling of host-parasite or individual-species systems. The introduction motivates the consideration of algebraic two-level measure trees. We will prove in this paper that the subspace of binary algebraic two-level measure trees can be embedded with a compact topology, allowing for limit theorems. We use this to construct the Kingman algebraic two-level measure tree from finite nested Kingman coalescent trees.\\

The common approach by now is to encode trees as metric spaces. For a metric space $(T,r)$ to have a tree-like structure, it is necessary to satisfy the so-called \emph{four-point condition} (see Figure~\ref{f:fourpointsshape}): for all $x_1,x_2,x_3,x_4\in T$,
\begin{equation}\label{e:4pointcondition}
	r(x_1,x_2)+r(x_3,x_4)\leq \max\big\{r(x_1,x_3)+r(x_2,x_4),r(x_1,x_4)+r(x_2,x_3)\big\}.
\end{equation}
However, not all metric spaces satisfying the four-point condition are tree-like, as for example the discrete triangle 3-graph (see Figure~\ref{f:fourpointsshape}, left). For this reason, \emph{metric trees} were introduced in \cite{AthreyaLoehrWinter2017} as are metric spaces $(T,r)$ satisfying the four-point condition \eqref{e:4pointcondition} and admitting branch points, i.e., for all $x_1,x_2,x_3 \in T$, there exists a (necessarily unique) $c_{(T,r)}(x_1,x_2,x_3)\in T$ such that
\begin{equation}
	r\big(x_i,c_{(T,r)}(x_1,x_2,x_3)\big)+r\big(c_{(T,r)}(x_1,x_2,x_3),x_j\big)=r(x_i,x_j)\quad \forall i,j\in\{1,2,3\}, i\neq j.
\end{equation}
Trees defined in the sense of graph theory can be equipped with the graph distance to obtain such metric trees. By letting the sizes of the trees grow to infinity, one might get, after a suitable rescaling of the distance, as a continuum limit a path-connected metric tree, which is referred to as $\R$-tree (see for example \cite{Tits1977,Dress1984,MayerOversteegen1990,MayerNikielOversteegen1992,DressMoultonTerhalle1996, DressTerhalle1996,Terhalle1997,Chiswell2001,Evans2008,AthreyaLoehrWinter2017}).

%It ensures that $T$ has a tree structure as there should only be one possible shape for the subtree spanned by four points (see Figure~\ref{f:fourpointsshape}). However, the assumption of connectedness does not allow for discrete trees. In , this condition was relaxed with the notion of 
%In particular, $\R$-trees admit branch points and each metric tree can be embedded isometrically into an $\R$-tree.

\begin{figure}[t]%-----------shape
\[
\xymatrix@=1pc{
&\bullet\ar@{-}[ddl]\ar@{-}[ddr]&&&&&		x_1\ar@{-}[dr]&&&&x_3\ar@{-}[dl]\\
&&&&&&										&c_1\ar@{-}[rr]&&c_2&\\
\bullet\ar@{-}[rr]&&\bullet&&&&				x_2\ar@{-}[ur]&&&&x_4\ar@{-}[ul]}\]
\caption{\josue{\textit{On the left:} the discrete triangle equipped with the graph distance satisfies the four-point condition \eqref{e:4pointcondition}, but does not admit branch points. \textit{On the right:} The only possible tree shape spanned by four points separates them into two pairs. Here, $c_1 = c(x_1 , x_2 , x_3 ) = c(x_1 , x_2 , x_4 )$ and $c_2 = c(x_1 , x_3 , x_4 ) = c(x_2 , x_3 , x_4 )$.}}
\label{f:fourpointsshape}
\end{figure}
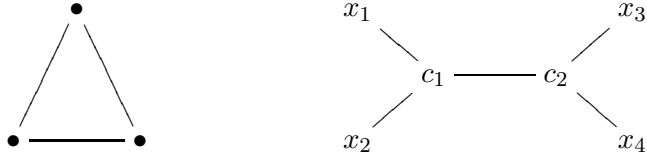

A new notion of potentially continuum trees, namely the algebraic trees, was introduced in \cite{LoehrWinter2021}. The focus is here shifted from the metric to the tree structure given by the so-called \emph{branch point map} which assigns to each triple of points their branch point. The algebraic trees are defined axiomatically by several conditions on the branch point map:

\begin{definition}[Algebraic tree]\label{d:algebraictree}
An \emph{algebraic tree} is a non-empty set $T$ together with a symmetric map $c\colon T^3\rightarrow T$ satisfying the following:
\begin{enumerate}
	\item[(2pc)] For all $x_1,x_2\in T$, $c(x_1,x_2,x_2)=x_2$.
	\item[(3pc)] For all $x_1,x_2,x_3\in T$, $c(x_1,x_2,c(x_1,x_2,x_3))=c(x_1,x_2,x_3)$.
	\item[(4pc)] For all $x_1,x_2,x_3,x_4\in T$,
	\begin{equation}
		c(x_1,x_2,x_3)\in\{c(x_1,x_2,x_4),c(x_1,x_3,x_4),c(x_2,x_3,x_4)\}.
	\end{equation}
\end{enumerate}
We call $c$ the \emph{branch point map}.
\end{definition}

As explained, this approach forgets the notion of metric (hence the denomination of \emph{algebraic trees} in contrast with \emph{metric trees}) and puts the emphasis on the tree structure. 
In order to sample leaves from an algebraic tree, we equip it with a Borel measure. \josue{For this, we consider on an algebraic tree $(T, c)$ the topology generated by so-called \emph{subtree components}, which are sets of the form
\begin{equation}
	\CS_x (y) := \{z \in T\setminus\{x\} : c(x, z, y)\neq x\},
\end{equation}
with $x,y\in T$.} An \emph{algebraic measure tree} $(T, c, \mu)$ consists of a separable algebraic tree $(T, c)$ together with a probability measure $\mu$ on the Borel $\sigma$-algebra $\CB(T,c)$. Associating each algebraic measure tree to the metric measure space \josue{where the intrinsic metric is defined with the help of the distribution of branch points (see Definition~\ref{d:bpd})}, we can use the Gromov-weak topology \josue{(see \cite{Gromov1999,GrevenPfaffelhuberWinter2009,Loehr2013})} to define a metrizable topology on the set $\T$ of (equivalence classes of) algebraic measure trees. Let $\T_2$ be the space of binary algebraic measure trees with no atoms on the skeleton:
\begin{equation}
	\T_2 := \big\{(T, c, \mu) \in \T : \text{degrees at most 3, atoms of }\mu\text{ only at leaves}\big\}.
\end{equation}
Due to the algebraic nature of the concept one might desire a more combinatorial notion of convergence. This can be done by considering the \emph{sample shape convergence} which, similarly to the Gromov-weak convergence, is based on a sampling procedure but evaluates the shape spanned by the sample rather than the matrix of mutual intrinsic distances. It was shown in \cite{LoehrWinter2021} that on $\T_2$ the sample shape convergence is equivalent to the (branch point distribution distance) Gromov-weak convergence and that for this topology, $\T_2$ is a compact subspace. 

This state space $\T_2$ has served to construct and study \josue{in \cite{LoehrMytnikWinter2020} the Aldous diffusion (\cite{Aldous2000}), and more generally in \cite{NussbaumerWinter2020} the $\alpha$-Ford diffusion (see \cite{Ford2005}), which are Markov processes on binary trees without edge lengths}. The compactness of $\T_2$ allows to get around tightness issues in these construction. Also, the sample shape convergence gives rise to a family of convergence determining classes of functions which are very useful when one wants to study tree-valued stochastic processes.\\

More recently, genealogical trees with two-level dynamics have been investigated. They are motivated by the study of two-level systems in biology, such as host-parasite, cell-virus or individual-species systems, where individuals of the first level are grouped together in clusters to form the second level and both levels are subject to resampling mechanisms. The phylogenies of such systems bring an important light on the paths and dynamical phenomenas that are interplaying. For example in anthropology or ecology these models have served as basis for statistical studies and understand collected data (e.g. \cite{BlumJakobsson2010,JayBoitardAusterlitz2019,LepersBilliardPorteMeleardTran2021,Verduetal2009}). For cell models, we can refer for example to \cite{Bansaye2008,BansayeTran2011,Kimmel1997,MeleardRoelly2013} among others.
For extend the notion of algebraic measure trees to hierarchical structures, one can equip metric spaces with two-level measures, that are measures on the set of measures. In \cite{Meizis2019}, the space of metric two-level measure spaces was introduced and equipped with the two-level Gromov-weak topology. It is a sampling topology which is consistent with a two-stage sampling procedure in host-parasite systems, where one would first sample hosts and then parasites within the hosts. The idea of representing a population with hierarchical structure by such a two-level measure is not new (see e.g.\ \cite{DawsonHochbergWu1990,Wu1991,Wu1994,GorostizaHochbergWakolbinger1995,DawsonHochbergVinogradov1995, DawsonHochbergVinogradov1996,GrevenHochberg2000,DawsonGorostizaWakolbinger2004,Dawson2018}). Metric two-level measure spaces allow for instance to study two-level tree-valued processes describing the evolution of ancestral relationships in hierarchical populations, where the metric encodes the genealogical distances between individuals. We applied this theory to define the two-level tree-valued Fleming-Viot dynamics in \cite{Nussbaumer2021}.

Our goal here is to adapt the approach of \cite{Meizis2019} for two-level trees in the metric setup to the algebraic one. An \emph{algebraic two-level measure tree} $(T, c, \nu)$ is thus defined as a separable algebraic tree $(T, c)$ together with a two-level measure $\nu\in\CM_1(\CM_1(T,c))$, i.e.\ a Borel probability measure on the set of Borel probability measures on $(T,c)$. \josue{Once more, we want to draw our attention away from the distances and focus more on the tree-shapes to define a combinatorial version of the two-level Gromov-weak topology.} In particular, we are interested in extending the results in \cite{LoehrWinter2021} to the space $\T_2^{(2)}$ of (equivalence classes of) binary algebraic two-level measure trees (see \eqref{def:T_2^2} for the proper definition).

\josue{When dealing with a two-level measure, say $\nu$, an important ingredient is the intensity measure $M_\nu\in\CM_1(T,c)$ defined by
\begin{equation*}
	M_\nu(\cdot):=\int\nu(\mathrm{d}\mu)\mu(\cdot).
\end{equation*}
Note that if $(T, c, \nu)$ is an algebraic two-level measure tree, $(T,c,M_\nu)$ is an algebraic measure tree. An algebraic two-level measure tree thus naturally defines a probability measure on $\T_2$ and we know that $\CM_1(\T_2)$ is compact since $\T_2$ is compact. However, it seems that the compactness of $\T_2$ cannot be directly used to prove the compactness of $\T_2^{(2)}$. Notice that the information contained in the two-level setup can not be summarized by a mean behaviour, in particular:
\begin{itemize}
	\item the map $(T, c, \nu)\mapsto(T,c,M_\nu)$ is continuous from $\T_2^{(2)}$ to $\T_2$, but it is not injective. To see this, consider for example, $T=\{x,y\}$, $c_{\{x,y\}}$ the only branch point map on $T$ and the two-level measures $\frac{1}{2}(\delta_{\delta_{x}}+\delta_{\delta_{y}})$ and $\delta_{\frac{1}{2}(\delta_{x}+\delta_{y})}$ which both have the same intensity measure $\frac{1}{2}(\delta_{x}+\delta_{y})$.
	\item the map that sends an a2m tree to the corresponding measure on $\T_2$ is not injective either. Indeed, $(\{x,y\},c,\delta_{x})$, $(\{x,y\},c,\delta_{y})$ and $(\{x\},c,\delta_{x})$ are equivalent in $\T_2$, so the measures on $\T_2$ associated to $(\{x,y\},c,\frac{1}{2}(\delta_{\delta_{x}}+\delta_{\delta_{y}}))$ and $(\{x\},c,\delta_{\delta_{x}})$ are equal.
\end{itemize}
}

\begin{figure}[t]
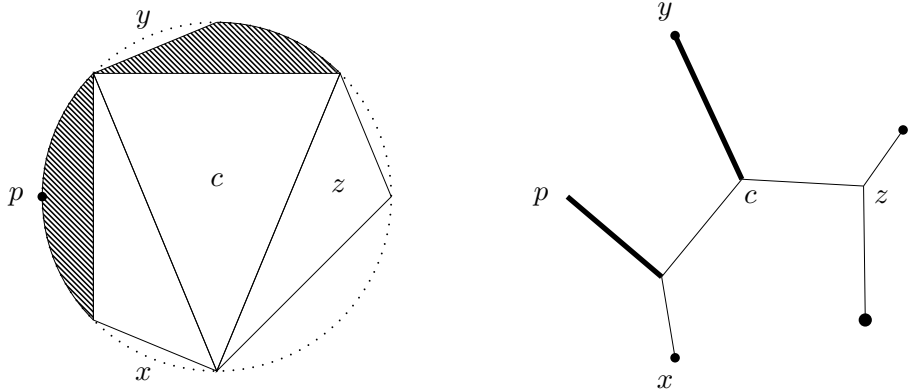

\begin{center}
\figureIexamplesubtriangulation
\end{center}
\caption{\emph{On the left:} A sub-triangulation of the circle with three empty triangles, four empty circular segments and two filled areas. \emph{On the right:} The tree coded by the sub-triangulation. The leaf $p$ does not carry an atom, but the four other leaves carry a weight given by the lengths of the corresponding arcs. The two thick segments carry a non-atomic mass.}\label{f:examplesubtriangulationintro}
\end{figure}

For this reason, we show the compactness of $\T_2^{(2)}$ by adapting a key ingredient used in \cite{LoehrWinter2021} to prove the compactness of $\T_2$, namely the coding of binary algebraic measure trees by \emph{sub-triangulations} of the circle. A similar encoding was first introduced by Aldous \cite{Aldous1994a,Aldous1994b}. In \cite{LoehrWinter2021}, a sub-triangulation of the circle $\BS$ is defined as a closed, non-empty subset $C$ of the disc satisfying the following two conditions:
\begin{enumerate}
	\item The complement of $C$ in the convex hull of $C$ consists of open interiors of triangles.
	\item $C$ is the union of non-crossing (non-intersecting except at endpoints), possibly degenerate closed straight line segments with endpoints in $\BS$.
\end{enumerate}
In this coding, branch points correspond to empty triangles, leaves carrying atoms to empty circular segments, and line segments with non-atomic mass to ``filled areas'' (see Figure~\ref{f:examplesubtriangulationintro}). Moreover, the arc lengths play an important role as they encode the way the mass is distributed in the algebraic measure tree. For example, a triangulation of an $n$-gon encodes the dual graph, equipped with the measure on the leaves given by the Lebesgue measure of the corresponding arcs of the circle line (see Figure~\ref{f:OLtreefromngon}).

We extend this coding in the two-level case. For that, one cannot simply rely on the Lebesgue measure of arcs anymore to encode all the information on the distribution of the random $\mu$-mass in the tree. We rather need to replace the Lebesgue measure by a two-level measure $K\in\CM_1(\CM_1(\BS))$ on the circle and we formally construct the \emph{coding map} that associates an algebraic two-level measure tree in $\T_2^{(2)}$ to a pair $(C,K)$ where $C$ is a sub-triangulation of the circle, and $K$ is a two-level measure on the circle line.

We show that the coding map is continuous and surjective when the set of sub-triangulations is equipped with the Hausdorff metric topology and the set of two-level measures on the circle line with the weak topology. Using that both of these topologies are compact, our main result (Theorem \ref{t:equivalencecompactness}) states that the space $\T_2^{(2)}$ is again compact. As an application, we finish with the construction of the Kingman algebraic two-level measure tree, which is the nested Kingman coalescent measure tree without branch length. This tree is the algebraic tree associated with the metric nested Kingman coalescent tree of \cite{Meizis2019}.\\

%\begin{figure}[t]
%\[
%\xymatrix@=1pc{
%u_{11}\ar@{-}[dr]&&  &&  &&\\
%&{\bullet}\ar@{-}[r]&u_{12},u_{22}  &&  u_{11},u_{21}\ar@{-}[rr]&&u_{12},u_{22}\\
%u_{21}\ar@{-}[ur]&&  &&  &&
%}\]
%\caption{For each algebraic tree, we take $\nu=\frac{1}{2}\left(\delta_{\frac{1}{2}(\delta_{u_{11}}+\delta_{u_{21}})}+\delta_{\frac{1}{2}(\delta_{u_{21}}+\delta_{u_{22}})}\right)$. These algebraic two-level measure trees are not equivalent. However, the associated measures on $\T_2$ are the same: they give probability 1 to the algebraic tree with two leaves, one edge, and the uniform distribution on the leaves.}
%\label{f:counterexample}
%\end{figure}

\begin{figure}[t]
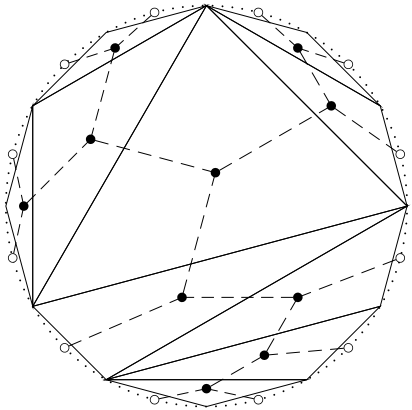

\begin{center}
\figureItreefromngon
\end{center}
\caption{A triangulation of the 12-gon. Here, the coded tree is the dual graph, with uniform distribution on the leaves.}\label{f:OLtreefromngon}
\end{figure}

\noindent\textbf{Outline.} We recall in Section~\ref{S:metricspaces} some definitions and results on metric one and two-level measure spaces. In Section~\ref{S:a2mtrees}, we introduce the space of (equivalence classes of) order separable algebraic measure trees, and equip it with a separable, metrizable topology based on the two-level Gromov-weak convergence of metric reprensatives. Section~\ref{S:triangulations} is devoted to the coding of binary algebraic two-level measure trees by sub-triangulations of the circle together with a two-level measure on the circle line. In Section~\ref{S:topologies}, we introduce with the two-level sample shape another topology on the subspace of binary algebraic two-level measure trees and we use the encoding of Section~\ref{S:triangulations} to show that both topologies are equivalent and compact. Finally, we apply this theory in Section~\ref{S:Kingman} to define the Kingman algebraic two-level measure tree.

\section{Metric one-level and two-level measure spaces}
\label{S:metricspaces}

We recall the basics on the theory of metric measure spaces in Section~\ref{s:mmspaces} and shortly explain how it is extended to a two-level setup in Section~\ref{s:m2mspaces}. This serves as a basis for constructing algebraic two-level measure trees.

\subsection{Metric measure spaces}
\label{s:mmspaces}

We introduce in this section the set of equivalence classes of metric measure spaces (see \cite{Gromov1999,GrevenPfaffelhuberWinter2009}). Equipped with the topology induced by so-called distance polynomials, it is a Polish space. Denoting by $\CM_1(X)$ the set of Borel probability measure on a metric space $X$, we start with the definition of a metric measure space.

\begin{definition}[Metric measure space]
A \emph{metric measure space (mm space)} $(X,r,\mu)$ is a non-empty Polish space $(X,r)$ together with a Borel probability measure $\mu\in\CM_1(X,r)$.
\end{definition}

Two metric measure spaces are equivalent if there exists a measure-preserving isometry between the supports of their respective measures is enough. Thus, by considering the equivalence class of a metric measure space $(X,r,\mu)$, one focuses on the structure of $\mu$ and on the restriction of $r$ to the support of $\mu$ denoted by $\supp(\mu)$, rather than the whole set $X$.

\begin{definition}[Equivalence of metric measure spaces]
\begin{enumerate}
	\item Two mm spaces $(X_i,r_i,\mu_i)$, $i=1,2$, are called \emph{mm-isomorphic} if there exists a measurable function $f\colon X_1\to X_2$ such that $\mu_2=f_{*}\mu_1$ and $f$ is isometric on the set $\supp(\mu_1)$ (but not necessarily on the whole space $X_1$). The function $f$ is called an \emph{mm-isomorphism}.
	\item The relation of being mm-isomorphic is an equivalence relation on the set of
metric measure spaces. \\
The set of equivalence classes of metric measure spaces is denoted by $\M$.
\end{enumerate}
\end{definition}

The following notion of test functions is based on the idea of sampling finite spaces of $(X,r)$ by means of $\mu$, that is, we sample a finite number of points and evaluate the matrix of their mutual distances.

\begin{definition}[Distance polynomials]
A \emph{distance polynomial} is a function $\Phi\colon \M \to \R$ of the form
\begin{equation}
	\Phi(\chi):=\int_{X^m}\mu^{\otimes m}(\du{u})\varphi\left(\big(r(u_{i},u_{j})_{i,j}\big)\right),
\end{equation}
where $\chi=(X,r,\mu)$, $m\in\N$, and $\varphi\in\CC_b(\R^{m^2})$. We write $\Pi_r$ for the set of all distance polynomials.
\end{definition}

The space $\M$ is then equipped with the coarsest topology such that all distance polynomials are continuous.

\begin{definition}[Gromov-weak topology]\label{d:gromovweak}
The \emph{Gromov-weak topology} is the initial topology on $\M$ induced by the test functions in $\Pi_r$. That is, a sequence of mm spaces $(\chi_n)_{n\in N}$ is said to \emph{converge Gromov-weakly} to $\chi$ in $\M$ if and only if $\Phi(\chi_n)$ converges to $\Phi(\chi)$ in $\R$, for all polynomials $\Phi\in\Pi_r$.
\end{definition}

With the following result, the space of metric measure spaces is suitable to do probability theory on it (see \cite[Theorem~1]{GrevenPfaffelhuberWinter2009}).

\begin{proposition}[$\M$ is Polish]
Equipped with the Gromov-weak topology, $\M$ is a Polish space.
\end{proposition}

\subsection{Metric two-level measure spaces}
\label{s:m2mspaces}

We here introduce the two-level analogues of the metric measure spaces from \cite{Meizis2019}. Equipped with the topology induced by two-level distance polynomials, the set of equivalence classes of metric two-level measure spaces is also Polish.

\begin{definition}[Metric two-level measure space]
A \emph{metric two-level measure space (m2m space)} $(X,r,\nu)$ is a non-empty Polish space $(X,r)$ together with a \emph{two-level measure} $\nu\in\CM_1(\CM_1(X,r))$, i.e.\ a Borel probability measure on the set of Borel probability measures on $(X,r)$.
\end{definition}

In order to define equivalence classes of m2m spaces, we introduce the so-called two-level push-forward operator. Let $(X, r)$ and $(\widehat{X}, \widehat{r})$ be Polish metric spaces and $g$ be a Borel measurable function from $X$ to $\widehat{X}$. As usual, $g_{*}\mu$ denotes the push-forward measure $\mu\circ g^{-1}$ for a Borel probability measure $\mu\in\CM_1(X)$. We regard $g_{*}$ as an operator
\begin{equation}
\begin{array}{ccccc}
	g_* & \colon & \CM_1(X) & \longrightarrow & \CM_1(\widehat{X})\\
	& & \mu & \longmapsto & g_{*}\mu.
\end{array}
\end{equation}
and call $g_{*}$ the \emph{(one-level) push-forward operator of $g$}. This enables us to define the \emph{two-level push-forward operator $g_{**}$ of $g$} by
\begin{equation}
\begin{array}{ccccc}
	g_{**} & \colon & \CM_1(\CM_1(X)) & \longrightarrow & \CM_1(\CM_1(\widehat{X}))\\
	& & \nu & \longmapsto & g_{**}\nu:=\nu\circ(g_*)^{-1}.
\end{array}
\end{equation}

We need yet another notion to define equivalence classes of m2m spaces. In the one-level case, the equivalence of two metric measure spaces is given by a measure-preserving isometry between the supports of their respective measures. For an m2m space, the measure $\mu$ is replaced by a two-level measure $\nu$ and the analog of $\supp(\mu)$ is the support of the \emph{intensity measure} $M_\nu\in\CM_1(X)$, also called \emph{first moment measure}, which is defined by
\begin{equation}\label{e:defMnu}
	M_\nu(\cdot)=\int\nu(\mathrm{d}\mu)\mu(\cdot).
\end{equation}
This notion allows to easily adapt some results on metric measure spaces to m2m spaces by replacing the (one-level) measure $\mu$ by the intensity measure of the two-level measure $\nu$ (compare for example \cite[Proposition~7.1]{GrevenPfaffelhuberWinter2009} to \cite[Theorem~7.2]{Meizis2019} in the case $\nu\in\CM_1(\CM_1(X))$). We will also exploit this idea in the case of algebraic trees below (e.g.\ in Definition~\ref{d:bpd}).

We are now able to define a notion of equivalence for m2m spaces.

\begin{definition}[Equivalence of m2m spaces]
\begin{enumerate}
	\item Two m2m spaces $(X_i,r_i,\nu_i)$, $i=1,2$, are called \emph{m2m-isomorphic} if there exists a measurable function $f\colon X_1\to X_2$ such that $\nu_2=f_{**}\nu_1$ and $f$ is isometric on the set $\supp(M_{\nu_1})$ (but not necessarily on the whole space $X_1$). The function $f$ is called an \emph{m2m-isomorphism}.
	\item The relation of being m2m-isomorphic is an equivalence relation on the set of
m2m spaces.\\
The set of equivalence classes of m2m spaces is denoted by $\M^{(2)}$.
\end{enumerate}
\end{definition}

The following notion of test functions is based on the idea of sampling finite spaces of $(X,r)$ by means of $\nu$, i.e.\ we first sample measures from $\CM_1(X)$ according to $\nu$ and then with each sampled measure, we sample finitely many points in $X$.

\begin{definition}[Two-level distance polynomials]
A \emph{two-level distance polynomial} is a function $\Phi\colon \M^{(2)} \to \R$ of the form
\begin{equation}
	\Phi(\chi):=\int\nu^{\otimes m}(\du{\mu})\int\bigotimes_{i=1}^m\mu_i(\du{u_i})\varphi\left(\big(r(u_{ij},u_{i'j'})_{(i,j),(i',j')}\big)\right),
\end{equation}
where $\chi=(X,r,\nu)$, $m\in\N$, $\vn\in\N^m$ and $\varphi\in\CC_b(\R^{|\vn|^2})$. We write $\Pi_r^{(2)}$ for the set of all two-level distance polynomials.
\end{definition}

The space $\M^{(2)}$ is then equipped with the coarsest topology such that all two-level distance polynomials are continuous.

\begin{definition}[Two-level Gromov-weak topology]\label{d:TLgromovweak}
The \emph{two-level Gromov-weak topology} is the initial topology on $\M^{(2)}$ induced by the test functions in $\Pi_r^{(2)}$. A sequence of m2m spaces $(\chi_n)_{n\in N}$ is said to \emph{converge two-level Gromov-weakly} to $\chi$ in $\M^{(2)}$ if and only if $\Phi(\chi_n)$ converges to $\Phi(\chi)$ in $\R$, for all polynomials $\Phi\in\Pi_r^{(2)}$.
\end{definition}

The following result is given by Proposition~4.6 and Theorem~8.1 in \cite{Meizis2019}:

\begin{proposition}[$\M^{(2)}$ is Polish]
Equipped with the two-level Gromov-weak topology, $\M^{(2)}$ is a Polish space.
\end{proposition}

\section{The space of algebraic two-level measure trees}
\label{S:a2mtrees}

In this section, we introduce the notion of algebraic two-level measure spaces. We first recall in Section~\ref{s:algebraictrees} tools on algebraic trees from \cite{LoehrWinter2021} on which we will rely in our set-up. In Section~\ref{s:a2mtrees}, we add two-level measures on algebraic trees and equip the space of algebraic two-level measure trees with a separable and metrizable topology.

\subsection{Algebraic trees}\label{s:algebraictrees} After recalling some basic examples and definitions,  we describe the natural topology that exists on an algebraic tree. We then present relations between algebraic trees and $\R$-trees that we will use to exploit results on metric two-level measure spaces. We finish with the notion of struture preserving morphisms to define a notion of equivalence on the space of a2m trees in Section~\ref{S:a2mtrees}.

Recall from Definition~\ref{d:algebraictree} the notion of algebraic tree given by the branch point map that sends a triplet of 3 points in the tree to its branch point. For an algebraic tree $(T,c)$, one can define for every $x,y\in T$, the \emph{interval} $[x,y]$ by
\begin{equation}
	[x,y]:=\{z\in T:c(x,y,z)=z\}.
\end{equation}
We say that $\{x,y\}$ is an \emph{edge} if $x\neq y$ and $[x,y]=\{x,y\}$, that is, there is ``nothing between $x$ and $y$'', and we denote by
\begin{equation}
	\edge(T,c)
\end{equation}
the set of edges of $(T,c)$. If $T$ is finite, then an algebraic tree $(T,c)$ simply is an undirected graph and $\edge(T,c)$ is the set of edges in the corresponding graph.

\begin{example}[Totally ordered spaces as algebraic trees]\label{e:totallyordered}
Let $(T,\leq)$ be a totally ordered space. For all $x,y,z\in T$ such that $x\leq y\leq z$, we define $c_\leq(x,y,z):=y$. It is easy to see that
$c_\leq$ is symmetric and satisfies the conditions (2pc)--(4pc), so that $(T,c_\leq)$ is an algebraic tree. Moreover, the interval $[x,y]$ in $(T,c_\leq)$ coincides with the order interval $\{z\in T:x\leq z\leq y\}$.
\end{example}

Conversely, if $(T,c)$ is an algebraic tree and $\rho\in T$ a distinguished point (often called the \emph{root}), we can define a \emph{partial order} $\leq_\rho$ by letting for $x,y\in T$,
\begin{equation}\label{e:partialorder}
	x\leq_\rho y\quad \Longleftrightarrow \quad x\in[\rho,y].
\end{equation}
This partial order allows us to define a notion of completeness of algebraic trees.

\begin{definition}[Directed order complete trees]
Let $(T,c)$ be an algebraic tree. We call $(T,c)$ \emph{(directed) order complete} if for all $\rho\in T$, the supremum of every totally ordered, non-empty subset exists in the partially ordered set $(T,\leq_\rho)$.
\end{definition}

The following definition gives the analogs of complete $\R$-trees, i.e., $\R$-trees that are complete as metric spaces.

\begin{definition}[Algebraic continuum tree]
We call an algebraic tree $(T, c)$ algebraic continuum tree if the following two conditions hold:
\begin{enumerate}
	\item $(T,c)$ is order complete.
	\item $\edge(T,c)=\emptyset$.
\end{enumerate}
\end{definition}

We can then equip an algebraic tree with a natural topology. To this end, we introduce, for each point $x\in T$, an equivalence relation $\sim_x$ on $T\setminus\{x\}$ such that for all $y,z\in T\setminus\{x\}$, $y\sim_xz$ if and only if $c(x,y,z)\neq x$. For $y\in T\setminus\{x\}$, we denote by
\begin{equation}
\label{e:components}
	\CS_x(y):=\{z\in T\setminus\{x\}:z\sim_x y\}
\end{equation}
the equivalence class of $y$ for this equivalence relation $\sim_x$. We also call $\CS_x(y)$ the \emph{component} of $T\setminus\{x\}$ containing $y$. The \emph{component topology} is defined as the one generated by the set of all components $\CS_x(y)$ with $x\neq y$, $x,y\in T$. In the following, we will suppose that algebraic trees are equipped with this topology.

\begin{figure}[t]%-----------shape
\[
\xymatrix@=1pc{
&{\bullet}\ar@{-}[d]&&{\bullet}&&y   &&&   &{\bullet}&&y\\
{\bullet}\ar@{-}[r]&{\bullet}\ar@{-}[dr]&&&{\bullet}\ar@{-}[ul]\ar@{-}[ur]&   &&&   &&{\bullet}\ar@{-}[ul]\ar@{-}[ur]&\\
&&x\ar@{-}[r]&{\bullet}\ar@{-}[r]&{\bullet}\ar@{-}[r]\ar@{-}[u]&{\bullet}   &&&   &{\bullet}\ar@{-}[r]\ar@{-}[l]&{\bullet}\ar@{-}[r]\ar@{-}[u]&{\bullet}\\
&{\bullet}\ar@{-}[ur]&&&{\bullet}\ar@{-}[u]&   &&&   &&{\bullet}\ar@{-}[u]&}\]
\caption{On the left: $x$ and $y$ are two points of the tree. On the right: the component $\CS_x(y)$.}
\label{f:components}
\end{figure}
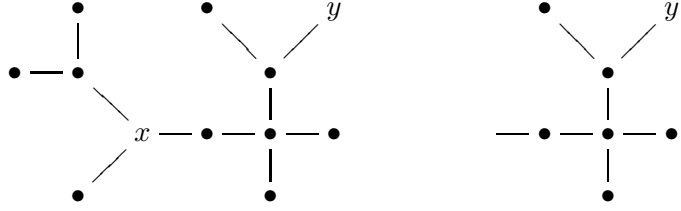

Furthermore, we will often assume the topology to be separable. But for many of our results, a condition on the number of edges will be crucial, so we rely on the following notion of separability (see also \cite[Example~2.23]{LoehrWinter2021}).

\begin{definition}[Order separability]
We call an algebraic tree $(T, c)$ \emph{order separable} if it is separable w.r.t.\ the component topology and has at most countably many edges.
\end{definition}

\begin{remark}[Sufficient condition for metrizability]\label{r:conditionmetrizable}
The component topology of any algebraic tree is Hausdorff \cite[Lemma~2.18]{LoehrWinter2021}. Thus, by Propositions~2.19 and 2.20 in \cite{LoehrWinter2021}, any order complete, order separable algebraic tree is a compact, second countable Hausdorff space. In particular, the component topology is metrizable and thus Polish.
\end{remark}

We also recall a definition of structure-preserving morphisms of algebraic trees which will allow to define equivalence classes of trees in Section~\ref{S:a2mtrees}.

\begin{definition}[Morphisms]\label{d:treemorphism} Let $(T, c)$ and $(\widehat{T}, \widehat{c})$ be algebraic trees. A map $f \colon T \to \widehat{T}$ is called a \emph{tree homomorphism} (from $T$ into $\widehat{T}$) if for all $x, y, z \in T$,
\begin{equation}
	f\big(c(x,y,z)\big)=\widehat{c}\big(f(x),f(y),f(z)\big),
\end{equation}
or equivalently, if for all $x,y\in T$,
\begin{equation}
	f([x,y])\subseteq[f(x),f(y)].
\end{equation}
We refer to a bijective tree homomorphism as \emph{tree isomorphism}.
\end{definition}

There exists a connection between algebraic trees and \emph{metric trees}, which are defined in \cite{AthreyaLoehrWinter2017} as metric spaces $(T,r)$ satisfying the 4-point condition \eqref{e:4pointcondition} and admitting branch points, i.e., for all $x_1,x_2,x_3$, there exists a (necessarily unique) $c_{(T,r)}(x_1,x_2,x_3)\in T$ such that
\begin{equation}
	r\big(x_i,c_{(T,r)}(x_1,x_2,x_3)\big)+r\big(c_{(T,r)}(x_1,x_2,x_3),x_j\big)=r(x_i,x_j)\quad \forall i,j\in\{1,2,3\}, i\neq j.
\end{equation}
Obviously, $c_{(T,r)}$ satisifies the conditions of Definition~\ref{d:algebraictree} and we call $(T,c_{(T,r)})$ the algebraic tree \emph{induced by} $(T,r)$, and $(T,r)$ a \emph{metric representation} of $(T,c_{(T,r)})$.

\begin{remark}[Homeomorphisms are tree homomorphisms]
Let $(T,r)$ and $(\widehat{T}, \widehat{r})$ be two $\R$-trees. Since the branch point map can be expressed in terms of intervals, a homeomorphism $f$ between $(T,r)$ and $(\widehat{T}, \widehat{r})$ is also a tree homomorphism between the corresponding induced algebraic trees (see \cite[Lemma~2.35]{LoehrWinter2021}).
\label{r:homeomorphisms}
\end{remark}

Conversely, under the assumption of order separability, one can build a metric representation of an algebraic tree $(T,c)$ as follows. Equip $(T,c)$ with the Borel $\sigma$-algebra $\CB(T, c)$ of the component topology. For any measure $\xi$ on $(T,\CB(T,c))$ such that $\xi$ is finite on every interval, we consider the following pseudometric
\begin{equation}\label{e:rxi}
	r_\xi(x,y) := \xi([x,y]) -\frac{1}{2} \xi(\{x\}) -\frac{1}{2}\xi(\{y\}),\quad x, y \in T.
\end{equation}

For such a measure $\xi$, denote by $T_\xi$ the set of equivalence classes of points in $T$ that are at distance zero for $r_\xi$. More precisely, $x$ and $y$ are equivalent if $\xi$ puts weight zero on the interval $[x,y]$. Denoting again the quotient metric on $T_\xi$ by $r_\xi$, $(T_\xi,r_\xi)$ is thus the quotient metric space. Let $\pi_\xi:T\to T_\xi$ be the canonical projection. We have the following \cite[Lemma~2.29]{LoehrWinter2021}.

\begin{lemma}[$(T_\xi, r_\xi)$ is a metric tree] Let $(T, c)$ be an algebraic tree, and $\xi$ a measure on $(T, c)$ with $\xi([x, y]) < \infty$ for all $x, y \in T$. Then the quotient space $(T_\xi, r_\xi)$ is a metric tree, and the canonical projection $\pi_\xi$ is a tree homomorphism.
\label{l:quotientspace}
\end{lemma}

Furthermore, if the measure $\xi$ satisfies $\xi[x,y]>0$ for all $x,y\in T$, then $r_\xi$ is a metric on $T$. Such a measure $\xi$ always exists in the case of order separable algebraic trees (see \cite[Lemma~2.32]{LoehrWinter2021}).

\subsection{Algebraic two-level measure trees}
\label{s:a2mtrees}

In this section, we define algebraic two-level measure trees (a2m trees for short) and equip the space of (equivalence classes of) a2m trees with a topology related to the two-level Gromov-weak topology for m2m spaces. We then give a result that allows to construct a2m trees from a partial knowledge of the two-level mass distribution.

Algebraic measure trees are introduced in \cite{LoehrWinter2021} as order separable algebraic trees equipped with probability measures. We extend this idea by equipping algebraic trees with two-level probability measures. The order separability condition is crucial in the sequel.

\begin{definition}[Algebraic two-level measure tree]
An \emph{algebraic two-level measure tree (a2m tree)} $(T, c, \nu)$ is an order separable algebraic tree $(T, c)$ together with a \emph{two-level measure} $\nu\in\CM_1(\CM_1(T,c))$, i.e.\ a Borel probability measure on the set of Borel probability measures on $(T,c)$.
\end{definition}

Similarly to m2m spaces, we will consider equivalence classes of a2m trees. Recall from \eqref{e:defMnu} the intensity measure $M_\nu$ and define a \emph{subtree} of $T$ as a set $A\subseteq T$ such that $c(A^3)=A$.

\begin{definition}[Equivalence of a2m trees]
\begin{enumerate}
	\item Two a2m trees $(T_i,c_i,\nu_i)$, $i=1,2$, are called \emph{a2m-isomorphic} if there exist subtrees $A_i$ of $T_i$ with $M_\nu(A_i)=1$ and a tree isomorphism $f\colon A_1\to A_2$ such that $\nu_2=f_{**}\nu_1$. The function $f$ is called an \emph{a2m-isomorphism}.
	\item The relation of being a2m-isomorphic is an equivalence relation on the set of
a2m trees. The set of equivalence classes of a2m trees is denoted by $\T^{(2)}$.
\end{enumerate}
\end{definition}

For an a2m tree $(T,c,\nu)$, let $A:=c((\supp(M_\nu))^3)$. By \cite[Corollary~2.3]{Meizis2019}, the support of $\nu$ is a subset of $\{\mu\in\CM_1(T)~:~\supp(\mu)\subseteq\supp(M_\nu)\}$. Thus the subtree $A$ of $(T,c)$ is such that $(T,c,\nu)$ is equivalent to $(A,c',\nu')$, where $c'$ is the restriction of $c$ to $A^3$ and $\nu'$ is the restriction of $\nu$ to $\CM_1(A,c')$. Therefore, we define for an a2m tree $\chi:=(T,c,\nu)$,
\begin{equation}
	\supp(\chi):=c((\supp(M_\nu))^3).
\end{equation}
With this in mind, note that by considering equivalence classes of a2m trees in the latter, we will only focus on the restriction of $c$ and $\nu$ to the support of $M_\nu$.

\josue{The example below is the analogue of \cite[Example~3.4]{LoehrWinter2021}, which states that there is only one equivalence class of algebraic measure trees without branch points nor atoms. For an algebraic tree $(T,c)$, we call $v\in T$ a \emph{branch point} if there exist $x_1,x_2,x_3\in T\setminus\{v\}$ such that $v=c(x_1,x_2,x_3)$. We denote by $\mathrm{br}(T,c)$ the set of branch points and if $\chi:=(T,c,\nu)$,
\begin{equation}
	\br(\chi):=\br(T,c)\cap\supp(\chi).
\end{equation}
We will also denote by $\at(\mu)$ the set of atoms of a measure $\mu$ on $(T,c)$. In particular, $\at(M_{\nu})$ are the atoms of the intensity measure $M_\nu$ defined in \eqref{e:defMnu}.}

\begin{example}[A2m trees without branch points nor atoms]\label{e:nobpnoat}
Let $\chi:=(T,c,\nu)$ be an a2m tree such that $\br(\chi)=\emptyset$ and $\at(M_\nu)=\emptyset$. By \cite[Theorem~1]{LoehrWinter2021}, there is a tree isomorphism from $T$ into $[0, 1]$ and we may thus assume $T \subseteq [0, 1]$. Let $F_{M_\nu}\colon[0,1]\to[0,1]$ be the cumulative distribution function of the intensity measure $M_\nu$. Define $\widetilde{\nu}:=(F_{M_\nu})_{**}\nu$. Then, $M_{\widetilde{\nu}}=F_{M_\nu*}M_\nu$ is the Lebesgue measure on $[0,1]$. Indeed since $\at(M_\nu)=\emptyset$, the function $F_{M_\nu}$ is continuous and for $a\in[0,1]$, the set $\{x:F_{M_\nu}\leq a\}$ is of the form $[0,x_a]$ with $F_{M_\nu}(x_a)=a$, so that $F_{M_{\widetilde{\nu}}}(a)=M_\nu(\{x:F_{M_\nu}\leq a\})=M_\nu([0,x_a])=F_\nu(x_a)=a$.

Let $A := \{x \in\supp(M_\nu)~:~\text{there is no } (y_n)_n\in([0, 1] \setminus \supp(M_\nu))^\N \mbox{such that} y_n <x, y_n \to x\}$ be the support of $M_\nu$ with left boundary points removed. Then $F_{M_\nu}$ restricted to $A$ is bijective and thus an a2m-isomorphism from $(A,c,\nu)$ onto $([0,1],c_{\leq},\widetilde{\nu})$ where $c_{\leq}$ is defined in Example~\ref{e:totallyordered}.

We showed that if we consider the equivalence class of an a2m tree $(T,c,\nu)$ such that $\br(\chi)=\at(M_\nu)=\emptyset$, we can always assume that $T=[0,1]$ and $M_\nu$ is the Lebesgue measure $\lambda_{[0,1]}$. More generally, if there exists an interval $(v,w)\subseteq \supp(M_\nu)$ such that $(v,w)\cap\at(M_\nu)=\emptyset$ and $(v,w)\cap\br(T,c)=\emptyset$, we can assume without loss of generality that $M_\nu$ restricted to $(v,w)$ is the Lebesgue measure $\lambda_{(v,w)}$.
\end{example}

We now equip the space $\T^{(2)}$ with a topology that relies on the two-level Gromov-weak topology on $\M^{(2)}$. To do so, we use that due to the order separability assumption, a2m trees allow for metric representations (see \cite[Theorem~1]{LoehrWinter2021}). But in order to get a useful topology on $\T^{(2)}$, we consider a particular metric representation of an a2m tree $(T, c, \nu)$ by using the metric $r_\xi$ defined in \eqref{e:rxi} with $\xi$ being the so-called \emph{branch point distribution}.

\begin{definition}[Branch point distribution]\label{d:bpd}
We call \emph{branch point distribution} of an a2m tree $(T, c, \nu)$ the push-forward of $M_\nu^{\otimes 3}$ under the branch point map,
\begin{equation}
	\xi := c_*M_\nu^{\otimes 3}.
\end{equation}
\end{definition}

Recall from Lemma~\ref{l:quotientspace} that the quotient space $(T_\xi,r_\xi)$ is a metric space. Therefore the following map associates a particular metric representation to each a2m tree.

\begin{definition}[Selection map $\iota$]
Define the \emph{selection map} $\iota\colon\T^{(2)}\to\M^{(2)}$ by
\begin{equation}
	\iota(T,c,\nu):=(T_\xi,r_\xi,\nu_\xi),
\end{equation}
where $\xi = c_*M_\nu^{\otimes 3}$ is the branch point distribution of $(T,c,\nu)$, $(T_\xi,r_\xi)$ is the quotient metric space, and $\nu_\xi:={\pi_\xi}_{**}\nu$ is the two-level push-forward measure of $\nu$ under the canonical projection $\pi_\xi$.
\end{definition}

It is easy to see that if two a2m trees are equivalent, then their images under $\iota$ are also equivalent, with the same isomorphism (see Remark~\ref{r:homeomorphisms}). Thus the selection map is well-defined. The following result states that $\iota$ indeed selects metric representations, and is injective.

\begin{proposition}[$\iota$ is an embedding]
For all $\chi\in\T^{(2)}$, $\iota(\chi)$ is a metric representation of $\chi$. Moreover, the selection map $\iota\colon\T^{(2)}\to\M^{(2)}$ is injective.
\end{proposition}

\begin{proof}
It is enough to show that for all $\chi\in\T^{(2)}$, the algebraic tree induced by $\iota(\chi)$ is $\chi$, because by Remark~\ref{r:homeomorphisms}, if $\iota(\chi)$ and $\iota(\chi')$ are equivalent in $\M^{(2)}$, the measure preserving bijective homeomorphism $f\colon\iota(\chi)\to\iota(\chi')$ is a tree homomorphism and thus yields an a2m-isomorphism on the corresponding a2m trees $\chi$ and $\chi'$.

Fix $\chi=(T,c,\nu)\in\T^{(2)}$. We can assume w.l.o.g that for all $v\in\mathrm{br}(T)$, $\xi\{v\}>0$. By Lemma~\ref{l:quotientspace}, the canonical projection $\pi_\xi\colon T\to T_\xi$ is a (surjective) tree homomorphism. Therefore, to show equivalence of $(T,c,\nu)$ and $(T_\xi,c_{(T_\xi,r_\xi)},\nu_\xi)$, it is sufficient to show that $\pi_\xi$ is injective on a subtree $A\subseteq T$ such that $M_\nu(A)=1$. We take $A:=T\setminus N$ with $N:=\{v\in T:\pi_\xi(v)\neq\{v\}\}$, and $\pi_\xi$  is injective on $A$.

Let us first show that $M_\nu(A)=1$. If $\pi_\xi(v)\neq\{v\}$ for some $v\in T$, then we can find some $u\neq v$ in $T$ such that $r_\xi(u,v)=0$, i.e., $\xi([u,v])-\frac{1}{2}\xi\{u\}-\frac{1}{2}\xi\{v\}=0$. Thus, $\xi\{v\}=0$ and $M_\nu\{v\}=0$. Moreover, since $\pi_\xi$ is a tree isomorphism, $w\in\pi_\xi(v)$ implies $[v, w] \subseteq \pi_\xi(v)$. But due to order separability, there are at most countably many non-degenerate, disjoint closed intervals in $T$, which implies that $\pi_\xi(N)$ is countable, and thus $M_\nu(N) = 0$.

Finally, to see that $A$ is a subtree, let $x,y,z\in A$. If $v:=c(x, y, z) \in \{x, y, z\}$, then $v\in A$. Otherwise, $v \in \mathrm{br}(T, c)$, and hence $M_\nu\{v\} > 0$ by assumption, which implies that $\pi_\xi(v) = \{v\}$, i.e.\ $v\in A$.
\end{proof}

Since the selection map is an embedding, it is suitable to define a useful topology on $\T^{(2)}$.

\begin{definition}[Two-level bpdd-Gromov-weak topology] \label{d:TLbpdd}
Let $\M^{(2)}$ be equipped with the two-level Gromov-weak topology. We call the topology induced on $\T^{(2)}$ by the selection map $\iota$ \emph{two-level branch-point distribution distance Gromov-weak topology (two-level bpdd-Gromov-weak topology)}.
\end{definition}

An immediate consequence of Proposition 2.6 is the following:

\begin{corollary}[Separability and metrizability]\label{p:separablemetrizable}
$\T^{(2)}$ equipped with the two-level bpdd-Gromov-weak topology is a separable, metrizable space.
\end{corollary}

\begin{proof}
The two-level Gromov-weak topology on $\M^{(2)}$ is separable and metrizable by the two-level Gromov-Prokhorov metric $d_{2\mathrm{GP}}$ (see \cite[Proposition~4.6, Theorem~8.1]{Meizis2019}). Now define for $\chi_1,\chi_2\in\T^{(2)}$,
\begin{equation}
	d_{2\mathrm{BGP}}(\chi_1,\chi_2) := d_{2\mathrm{GP}}(\iota(\chi_1),\iota(\chi_2)).
\end{equation}
Since $\iota$ is injective, $d_{2\mathrm{BGP}}$ is a metric on $\T^{(2)}$ and it induces the two-level bpdd-Gromov-weak topology.
\end{proof}

\section{Triangulations of the circle}
\label{S:triangulations}

A key point in proving compactness of $\mathbb{T}_2$ in \cite{LoehrWinter2021} was to encode binary algebraic measure trees as sub-triangulations of the circle, and then exploit the compactness of the space of sub-triangulations equipped with the Hausdorff topology. We are interested in formulating a similar result in the case of binary a2m trees. In order to encode the random masses of the subtree components, we add a two-level measure on the circle to the sub-triangulation.

We first define what we mean by sub-triangulations of the circle and in Subsection~\ref{s:coding}, we construct the map that associates an a2m tree to every sub-triangulation together with a two-level measure on the circle line and show that, as in the one-level case, this coding map is continuous and surjective.

\subsection{The space of sub-triangulations of the circle}

In the whole Section~\ref{S:triangulations}, we fix $\D$ a closed disc of circumference 1, and denote by $\BS := \partial\D$ the circle. We will repeatedly identify $\BS$ with $[0,1]$ where the endpoints are glued. In this section, we denote by $\lambda_I$ the Lebesgue measure on an interval $I\subseteq \BS\simeq[0,1]$. For a subset $A \subseteq \D$, we denote by $\overline{A}$, $A^\circ$, $\partial A$ and $\conv(A)$ the closure, the interior, the boundary and the convex hull of $A$, respectively. Furthermore, we define
\begin{equation}\label{def:t1}
	\Delta(A):=\big\{\text{connected components of }\conv(A)\setminus A\big\},
\end{equation}
and
\begin{equation}\label{def:t2}
	\nabla(A):=\big\{\text{connected components of }\D\setminus \conv(A)\big\},
\end{equation}
so that we have the partition of the disc
\begin{equation}
	\D=A\uplus\biguplus_{a\in\Delta(A)}a\uplus\biguplus_{b\in\nabla(A)}b,
\end{equation}where $\biguplus$ denotes the disjoint union.

\begin{definition}[Sub-triangulations of the circle]
A \emph{sub-triangulation of the circle} is a closed non-empty subset $C$ of $\D$ satisfying the following two conditions:
\begin{enumerate}
	\item[(Tri1)] $\Delta(C)$ consists of open interiors of triangles.
	\item[(Tri2)] $C$ is the union of non-crossing (non-intersecting except at endpoints), possibly degenerate closed straight line segments with endpoints in $\BS$.
\end{enumerate}
We denote the set of sub-triangulations of the circle by $\CT$.
\end{definition}

Note that given (Tri1), (Tri2) implies that $\nabla(C)$ consists of circular segments with the bounding straight line excluded and the rest of the bounding arc included (see Figure~\ref{f:examplesubtriangulation}).

\begin{figure}[t]
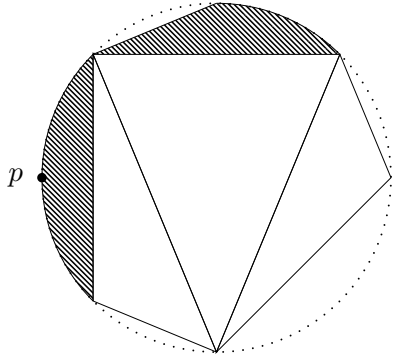

\begin{center}
\figureexamplesubtriangulation
\end{center}
\caption{A sub-triangulation $C$ of the circle. From the definitions \eqref{def:t1}, \eqref{def:t2}, \eqref{def:t3} and \eqref{def:t4}, we have in $C$: $\#\Delta(C)=3$, $\#\nabla(C)=4$, $\#\blacksquare(C)=2$ and $\#\square(C)=1$, (more precisely, $\square(C)=\{p\}$).}\label{f:examplesubtriangulation}
\end{figure}

The set $\CT$ can be equipped with a compact metrizable topology in the following way. Let
\begin{equation}
	\CF(\D):=\big\{F\subseteq\D\ :\ F\neq\emptyset,\ F \mbox{closed}\big\}.
\end{equation}
We equip $\CF(\D)$ with the \emph{Hausdorff metric topology}. That is, we say that a sequence $(F_n)_{n\in\N}$ converges to $F$ in $\CF(\D)$ if and only if for all $\epsilon > 0$ and all large enough $n\in\N$,
\begin{equation}
	F\subseteq F_n^\epsilon \quad \text{ and } \quad F_n\subseteq F^\epsilon,
\end{equation}
where for all $A \in F(\D)$, $A^\epsilon:=\{x\in\D:d(x,A)<\epsilon\}$. As $\D$ is compact, $\CF(\D)$ is a compact metrizable space which contains $\CT$ and it can be shown that $\CT$ is actually a closed subset of $\CF(\D)$ (see \cite[Lemma~4.2]{LoehrWinter2021}).

In order to construct the branch point map from a sub-triangulation in the next subsection, we need another characterization of sub-triangulations of the circle. Roughly speaking, condition (Tri2) can be replaced by the existence of triangles that separate triples of connected components of $\D\setminus C$.

\begin{proposition}[Characterization of sub-triangulations]\label{p:characsubtriangulation}
Let $C$ be a closed non-empty subset of $\D$. Then $C$ is a sub-triangulation of the circle if and only if condition (Tri1) holds, all extreme points of $\conv(C)$ are contained in $\BS$ and
\begin{enumerate}
	\item[(Tri2)'] For $x, y, z \in \Delta(C)\cup\nabla(C)$ pairwise distinct, there exists a unique $c_{xyz}\in\Delta(C)$ such that $x, y, z$ are subsets of pairwise different connected components of $\D\setminus\partial c_{xyz}$.
\end{enumerate}
\end{proposition}

\subsection{Coding of binary algebraic two-level measure trees}\label{s:coding}
We define here a map that associates a binary a2m tree to a sub-triangulation together with a two-level measure on the circle line whose intensity measure is the Lebesgue measure. We show that this coding map is surjective and continuous.

\josue{We start by defining the set of binary algebraic two-level measure trees with no atoms on the skeleton. For this, define in an algebraic $(T,c)$ the \emph{degree} of $x\in T$ as the number of components of $T\setminus\{x\}$, and we write
\begin{equation}
	\mathrm{deg}(x):=\#\{\CS_x(y):y\in T\setminus\{x\}\}.
\end{equation}
A \emph{leaf} is then a point $u\in T$ such that $\mathrm{deg}(u)=1$, and we write $\mathrm{lf}(T,c)$ for the set of leaves. We then define}
\begin{equation}\label{def:T_2^2}
	\T_2^{(2)} := \left\{\chi=\overline{(T,c,\nu)}\in\T^{(2)}:\deg(v)\leq 3~\forall v\in T,~\at(\mu)\subseteq\lf(T,c)~\forall\mu\in\supp(\nu)\right\}.
\end{equation}
Note that for all $(T,c,\nu)\in\T_2^{(2)}$, we also have $\at(M_\nu)\subseteq\lf(T,c)$.

In Theorem~\ref{t:codingmap}, we introduce a coding map which associates an a2m tree $(T,c,\nu)\in\T_2^{(2)}$ to each two-level sub-triangulation, which are defined below:

\begin{definition}[Two-level sub-triangulation]
A \emph{two-level sub-triangulation of the circle} $(C,K)$ consists of a sub-triangulation of the circle $C$ together with a two-level measure $K\in\CM_1(\CM_1(\BS))$ on $\BS$ such that its intensity measure $M_K$ is $\lambda_\BS$ the Lebesgue on the circle line.

We denote by
\begin{equation}\label{e:TLsubtri}
	\mathfrak{D} := \big\{(C,K)\in\CT\times\CM_1(\CM_1(\BS)):M_K=\lambda_\BS\big\}
\end{equation}
the set of all two-level sub-triangulations of the circle.
\end{definition}

In the construction we give below, the algebraic measure tree $(T,c,M_\nu)$ is the one associated to the sub-triangulation $C$ given by the coding map in \cite[Theorem~2]{LoehrWinter2021}. In particular, a triangle $x\in\Delta(C)$ will correspond to a branch point of the associated a2m tree and a circular segment $x\in\nabla(C)$ to an atom of the intensity measure $M_\nu$ of the a2m tree such that the arc $x\cap\BS$ has length the $M_\nu$-mass of the corresponding atom.

\begin{example}[A simple case]
Consider a triangulation $C$ of the regular $n$-gon into $n-2$ triangles (see Figure~\ref{f:treefromngon}). In this case, the coded tree is the dual graph. That is, a triangle correponds to a branch point of the tree, and two branch points of the tree are connected by an edge if and only if the triangles are adjacent. Each circular segment corresponds to a leaf and $M_\nu$ will assign to each leaf the length of the circular segment, i.e.\ $n^{-1}$ since the $n$-gon is regular.

We then add on the circle line a two-level measure $K$ such that $M_K=\lambda_\BS$ by defining
\begin{equation}
	K=\frac{1}{2}\delta_{\kappa_1}+\frac{1}{2}\delta_{\kappa_2},
\end{equation}
where $\kappa_1,\kappa_2\in\CM_1(\BS)$ are defined in Figure~\ref{f:treefromngon}. Then the associated two-level measure $\nu$ on the tree is given by $\frac{1}{2}\delta_{\mu_1}+\frac{1}{2}\delta_{\mu_2}$ where $\mu_i$ assigns to each leaf the $\kappa_i$-mass carried by the corresponding circular segment.
\end{example}

\begin{figure}[t]
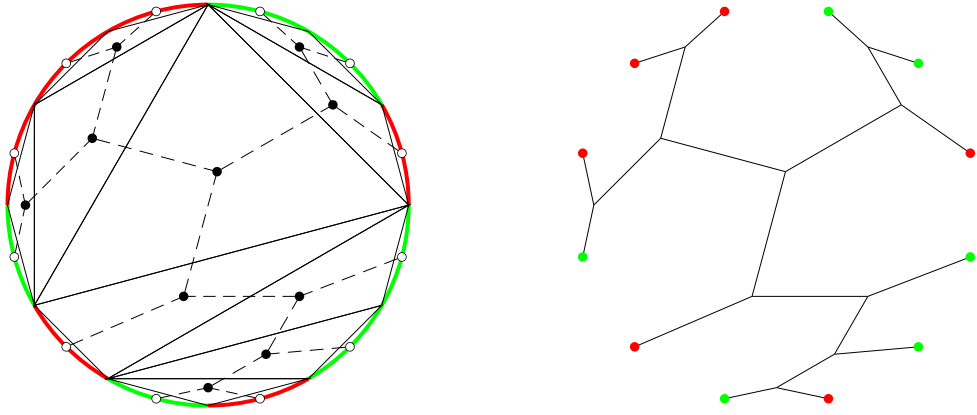

\begin{center}
\figuretreefromngon
\end{center}
\caption{A triangulation of the 12-gon and the tree coded by it. Suppose that we add on the triangulation the two-level measure $K=\frac{1}{2}\delta_{\kappa_1}+\frac{1}{2}\delta_{\kappa_2}$, where $\kappa_1$ is the renormalized Lebesgue measure on the union of the red circular segments and $\kappa_2$ on the union of the green circular segments. Then the corresponding two-level measure on the tree is given by $\frac{1}{2}\delta_{\mu_1}+\frac{1}{2}\delta_{\mu_2}$ where $\mu_1$ assigns mass $\frac{1}{6}$ to each red leaf and $\mu_2$ to each green leaf.}\label{f:treefromngon}
\end{figure}

\begin{remark}[About the condition $M_K=\lambda_\BS$]
Consider the sub-triangulation $C$ and $K:=\delta_{\delta_0}$ as in Figure~\ref{f:conditionMK} on the left. We need a condition on $K$ to avoid such a case where $M_K$ puts an atom on an endpoint of a line segment of $C$ separating two circular segments, \josue{as we would not know where to assign the mass in the tree}. We could have overcome this issue by deciding \emph{a priori} to which of the two circular segment the atom adds mass. However, with this solution, the coding map would not be continuous when $\CT$ is equipped with the Hausdorff metric topology and $\CM_1(\CM_1(\BS))$ with the weak topology. Indeed, suppose we decide to add the atoms on endpoints to the mass of the circular segment after the endpoint and let $K_n:=\delta_{\frac{1}{2}\delta_0+\frac{1}{2}\delta_{-n^{-1}}}$ (see Figure~\ref{f:conditionMK} on the right). Then the sequence $(C,K_n)$ converges to $(C,K)$ but it is not true of the corresponding a2m trees since the one associated with $(C,K_n)$ has two leaves carrying mass but not the one associated with $(C,K)$.

The condition $M_K=\lambda_\BS$ prevents such cases from occurring and it seems natural for the specific reason that it also ensures that the algebraic measure tree $(T,c,M_\nu)$ corresponding to $(C,K)$ is the one associated to the sub-triangulation $C$ in the one-level case.
\end{remark}

\begin{figure}[t]
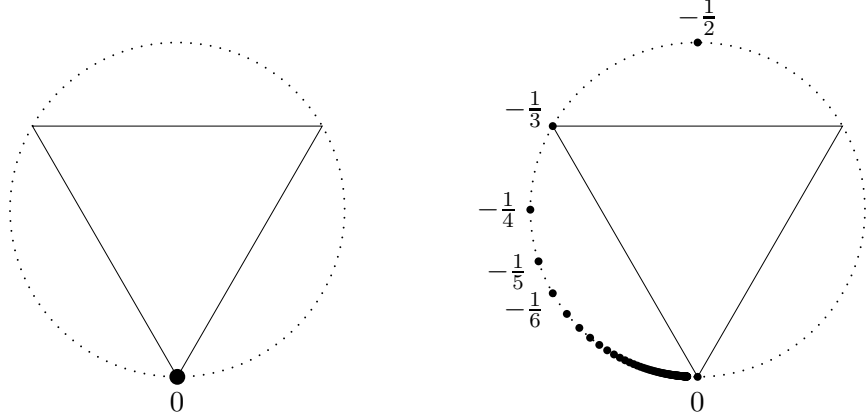

\begin{center}
\figureconditionMK
\end{center}
\caption{Examples of two-level measures $K$ on the circle such that $M_K\neq\lambda_\BS$. On the left, $K=\delta_{\delta_0}$ and on the right, we represent $K_n=\delta_{\frac{1}{2}\delta_0+\frac{1}{2}\delta_{n^{-1}}}$ for several $n\in\N$. Then $K_n$ converges weakly to $K$. If we add the atom on endpoints to the mass of the circular segment after the endpoint, the tree associated with $(C,K_n)$ has two leaves carrying mass but the one associated with $(C,K)$ is a sinle point carrying mass. Therefore the sequence of trees corresponding to $(C,K_n)$ do not converge to the one associated with $(C,K)$.}\label{f:conditionMK}
\end{figure}

Before stating the result, we need further notation. We first define on the sub-triangulation the points that will correspond through the coding construction to leaves without mass. For $x\in\Delta(C)$, let $p_i(x)$, $i=1,2,3$, be the mid-points of the three arcs of $\BS\setminus\partial x$ \josue{and let $\mathrm{arc}_x(p_i(x))$ be the arc of $\BS\setminus\partial x$ that contains $p_i(x)$}. Define
\begin{equation}\label{def:t3}
	\square(C):=\big\{\{p_i(x)\}:x\in\Delta(C),i\in\{1,2,3\},\mathrm{arc}_x(p_i(x))\subseteq C\big\}
\end{equation}
which is the set of mid-points of ``filled'' circular segments for $C$ (see Figure~\ref{f:examplesubtriangulation}). \\
In the tree, the corresponding leaves will be connected to the rest of the tree through line segments with non-atomic measure. For $(T,c,\nu)$ an algebraic measure tree, we call an interval and denote it by $(v,w)$ a subset of $\supp(M_\nu)$ such that $(v,w)\cap\at(M_\nu)=\emptyset$ and $(v,w)\cap\br(T,c)=\emptyset$ a \emph{line segment with non-atomic measure}. We denote by
\begin{equation}
	\seg(T,c,\nu)
\end{equation}
the set of maximal (w.r.t.\ inclusion) line segments with non-atomic measure. Note that if $(T,c)$ is order complete, every line segment with non-atomic measure is included in some $(v,w)\in\seg(T,c,\nu)$. \\
The analog of these line segments in the sub-triangulation are ``filled'' areas, that is, elements of
\begin{equation}\label{def:t4}
	\blacksquare(C):=\big\{\overline{b}~:~b\text{ is a connected component of }C^\circ\big\}.
\end{equation}

For such a filled area $b\in\blacksquare(C)$, it will be important to be able to know the mass distribution of $K$ along $b\cap\BS$, and not only the total mass carried by $b\cap\BS$. For this reason, we use that by definition of a sub-triangulation, $b\in\blacksquare(C)$ is the union of non-crossing closed straight line segments with endpoints in $\BS$. However, there are several possible ways to decompose $b$ into infinitely many line segments. For the sake of simplicity, we choose to partition $b$ ``linearily'' in the following sense. Recall that we identify $\BS$ with $[0,1]$ where the endpoints are glued. If $b\cap\BS$ is not connected, then it is the union of two circular segments $[x,y]$ and $[y',x']$ (with $x=y$ or $x'=y'$ possibly, see Figure~\ref{f:partitionfilledareas}, left). If $b\cap\BS$ is connected, it is a circular segment $[x,x']$ and we take $y=y'$ to be its mid-point (which belongs to $\square(C)$), so that $b\cap\BS$ is still the union of two circular segments (see Figure~\ref{f:partitionfilledareas}, right). We then define $\|_b(C)$ as the set of all straight line segments connecting a point of $[x,y]$ to a point of $[y',x']$ in a linear way (see Figure~\ref{f:partitionfilledareas}). For example, if $x\leq y\leq y'\leq x'$,
\begin{equation}\label{e:partitionofb}
	\|_b(C):=\left\{[(1-t)x+ty,(1-t)x'+ty']~\big|~t\in[0,1]\right\}.
\end{equation}

Define
\begin{equation}
	\|(C):=\bigcup_{b\in\blacksquare(C)}\|_b(C).
\end{equation}

\begin{figure}[t]
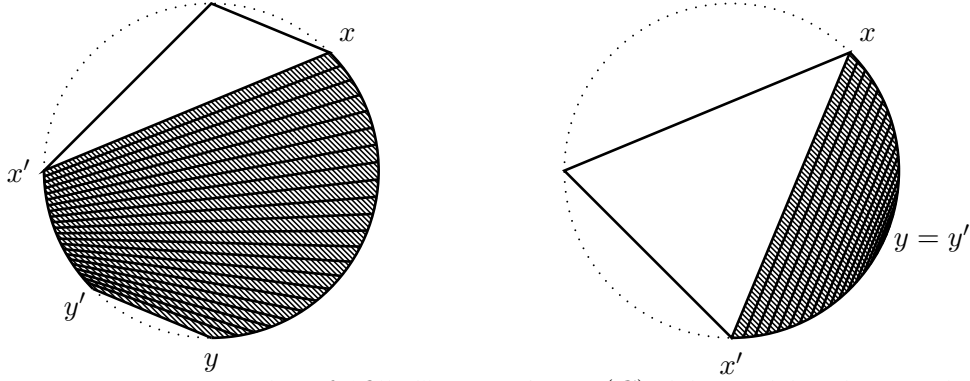

\begin{center}
\figurefilledareas
\end{center}
\caption{Two examples of ``filled'' areas $b\in\blacksquare(C)$ delimited by the circular segments $[x,y]$ and $[x',y']$. They are partitioned into straight line segments with endpoints in $[x,y]$ and $[x',y']$.}\label{f:partitionfilledareas}
\end{figure}

Finally, we define for the sub-triangulation the analog of the components in an algebraic tree. For $x\in\Delta(C)\cup\nabla(C)$, and $y\subseteq\D$ connected and disjoint from $\partial x$, let 
%$\partial_\D x$, where $\partial_\D$ denotes the boundary in the space $\D$, let
\begin{equation}
	\comp_x(y):= \text{the connected component of }\D\setminus\partial x\text{ which contains }y.
\end{equation}
Similarly, for $x\in\|(C)$ and $y\subseteq\D$ connected and disjoint from $x$, let
\begin{equation}
	\comp_x(y):= \text{the connected component of }\D\setminus x\text{ which contains }y.
\end{equation}
We also call them \emph{components} (of the sub-triangulation) (see Figure~\ref{f:definitioncomp}). Define also $\comp_x(x):=x$ for $x\in\|(C)\cup\square(C)$.

\begin{figure}[t]
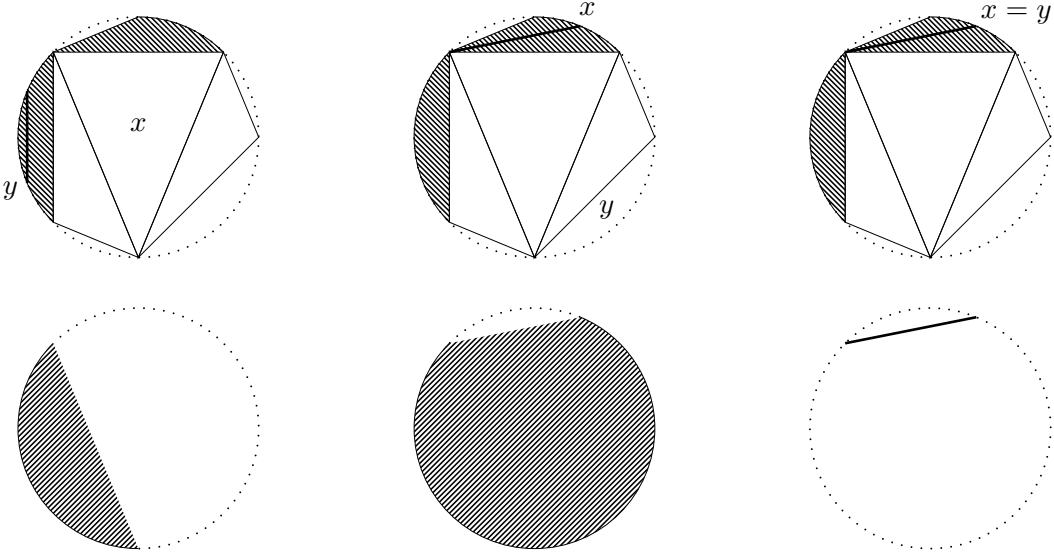

\begin{center}
\figuredefinitioncomp
\end{center}
\caption{Three examples of components $\comp_x(y)$ in a sub-triangulation of the circle $C$. The top line represents the triangulation, $x$ and $y$. The second line represents the corresponding components $\comp_x(y)$. On the left, $x\in\Delta(C)$ and $y\in\|(C)$. In the middle, $x\in\|(C)$ and $y\in\nabla(C)$. In these two cases, $\comp_x(y)$ is a filled cap region of the disc with the boundary straight line excluded and the rest of the boundary arc included. On the right, $x=y\in\|(C)$ and $\comp_x(x)=x$. }\label{f:definitioncomp}
\end{figure}

The following lemma gives the tree associated to a triangulation $C$. Contrary to \cite[Lemma~4.7]{LoehrWinter2021}, we add the set $\|(C)$ to the ``skeleton'' of the tree as it is important in the two-level case to know how the mass is distributed along a line segment of the tree. \josue{For $(T,c)$ an algebraic tree and $x\in T$, we first extend the notation \eqref{e:components} for components to
\begin{equation}
	\CS_x(x):=\{x\}.
\end{equation}}

\begin{lemma}[induced branch-point map]
For $C\in\CT$, let $V_C:=\Delta(C)\cup\nabla(C)\cup\square(C)\cup\|(C)$. If $V_C\neq \emptyset$, then there exists a unique branch-point map $c_V\colon V_C^3\to V_C$, such that $(V_C,c_V)$ is an algebraic tree with
\begin{equation}
	\CS_x^{(V_C,c_V)}(y)=\{v\in V_C:\comp_x(y)=\comp_x(v)\}
\end{equation}
for $x,y\in V_C$. In particular, $\deg(x)=3$ for all $x\in\Delta(C)$, $\deg(x)=1$ for all $x\in\nabla(C)\cup\square(C)$ and $\deg(x)=2$ for all $x\in\|(C)\setminus\square(C)$.
\label{l:inducedbpmap}
\end{lemma}

\begin{proof}
Condition (Tri2)' of Proposition~\ref{p:characsubtriangulation} gives that for pairwise distinct $x, y, z \in\Delta(C)\cup\nabla(C)$, there is a unique triangle $c_{xyz}\in\Delta(C)$ such that $x, y, z$ are subsets of pairwise different connected components of $\D\setminus\partial c_{xyz}$. This can be extended to triples of points in $V_C$ in an obvious way. It is then easy to see that this defines a branch point map on $V_C$.
\end{proof}

The following theorem is the analog of \cite[Theorem~2]{LoehrWinter2021}. It states that each pair $(C,K)$, where $C$ is a sub-triangulation of the circle and $K$ a two-level measure on $\BS$ with $M_K=\lambda_\BS$, can be associated with a binary a2m tree such that $\Delta(C)$ corresponds to the set of branch points and $\nabla(C)$ corresponds to the set of atoms of the intensity measure of $\nu$ in the tree. Furthermore, $\comp_v(w)$ corresponds to the component $\CS_v(w)$ and its random $\nu$-mass is given by the random $K$-mass carried by $\comp_v(w)\cap \BS$.

\begin{theorem}[Coding map]\label{t:codingmap}
\begin{enumerate}
	\item[(i)] For all $\Gamma=(C,K)\in\mathfrak{D}$, there is a unique (up to equivalence) a2m tree $\chi_\Gamma=(T_\Gamma,c_\Gamma,\nu_\Gamma)\in\T^{(2)}_2$ such that:
	\begin{enumerate}
		\item[(CM1)] $V_C\subseteq T_\Gamma$, $\mathrm{br}(T_\Gamma,c_\Gamma)=\Delta(C)$, and $c_\Gamma$ is an extension of $c_V$, where $(V_C,c_V)$ is defined in Lemma~\ref{l:inducedbpmap}.
		\item[(CM2)] For all $x,y\in V_C$ and $\sigma\geq0$,
		\begin{equation}
			\int\nu_\Gamma(\mathrm{d}\mu)e^{-\sigma\mu(\CS_x(y))}=\int K(\mathrm{d}\kappa)e^{-\sigma\kappa(\comp_x(y)\cap\BS)}.
		\end{equation}
		\item[(CM3)] $\mathrm{at}(M_{\nu_\Gamma})=\nabla(C)$.
	\end{enumerate}
	\item[(ii)] The \emph{coding map} $\tau:\mathfrak{D}\to\T_2^{(2)},~\Gamma\mapsto\chi_\Gamma$ is surjective.
	\item[(iii)] Let $\CT$ be equipped with the Hausdorff metric topology, $\CM_1(\CM_1(\BS))$ with the weak topology, $\CT\times\CM_1(\CM_1(\BS))$ with the product topology and $\T_2^{(2)}$ with the two-level bpdd-Gromov-weak topology. Then the coding map $\tau$ is continuous.
\end{enumerate}
\end{theorem}

\josue{We first state an extension result from \cite{LoehrWinter2021} for algebraic measure trees that will be a key ingredient in the proof. Recall that for $(T,c)$ an algebraic tree and $x\in T$, we have $\CS_x(x):=\{x\}$ so that $T$ is the disjoint union of the $\deg(x)+1$ sets in
\begin{equation}
	\CC_x := \big\{\CS_x(y)\ :\ y\in T\big\}.
\end{equation}
For $y\in T$ and $V\subseteq T$, we call a function $f\colon V\to \R$ \emph{order-left continuous} on $V$ with respect to the partial order $\leq_y$ (defined by \eqref{e:partialorder}) if for all $x,x_n\in V$ such that $x_n\uparrow_y x$, we have $\lim_{n\to\infty}f(x_n)=f(x)$, where $x_n\uparrow_y x$ means that $x_1\leq_y x_2\leq_y \cdots$ and $x=\sup_{n\in\N}x_n$ with respect to $\leq_y$.

From \cite[Proposition~2.20]{LoehrWinter2021}, we know that if an algebraic tree $(T,c)$ is order separable, there exists a subset $V$ of $T$ that is \emph{order dense}, i.e., such that for all $x, y \in T$ with $x \neq y$,
\begin{equation}\label{e:orderdenseset}
	D\cap[x,y)\neq\emptyset.
\end{equation}
The following result \cite[Proposition~3.12]{LoehrWinter2021} allows to extend the function that associates to subtrees their masses to a measure on the algebraic tree. 

\begin{proposition}[Extension to a measure]
Let $(T, c)$ be an order separable algebraic continuum tree, and $V \subseteq T$ order dense. Then a set-function $\mu_0 : \CC_V := \bigcup_{x\in V} \CC_x \to [0, 1]$ has a unique extension to a probability measure on $\CB(T, c)$ if it satisfies
\begin{enumerate}
	\item For all $x\in V$, $\sum_{A\in\CC_x}\mu_0(A)=1$.
	\item For all $x,y\in V$ with $x\neq y$,
	\begin{equation}
		\mu_0(\CS_x(y))+\mu_0(\CS_y(x))\geq 1.
	\end{equation}
	\item For every $y\in V$, the function $\psi_y:x\mapsto \mu_0(\CS_x(y))$ is order left-continuous on $V$ with respect to $\leq_y$.
\end{enumerate}
\label{p:extmeasure}
\end{proposition}}

Let us now sketch the proof of Theorem~\ref{t:codingmap} which is proved below. (i) We start by extending the tree $(V_c,c_V)$ to an algebraic continuum tree and apply Proposition~\ref{p:extmeasure} ``almost surely'' to construct the two-level measure $\nu$. (ii) We give the proof of the surjectivity of the coding map $\tau$ in three steps. For $\chi=(T,c,\nu)\in\T_2^{(2)}$,
\begin{enumerate}
	\item we first construct the sub-triangulation $C$ associated with the algebraic measure tree $(T,c,M_\nu)$ as in the proof of \cite[Theorem~2]{LoehrWinter2021}. That is, for each branch point we remove an open triangle from the set $\D$ and for each atom of $M_\nu$, we remove a circular segment.
	\item we then build $K$ by weak approximation. For this, we rely on a result from \cite{Meizis2019} that the two-level measure $\nu$ can be reconstructed from an infinite set of points randomly sampled. The two-level measure $K$ is then built in a similar fashion by using a correspondance between points in the sample and subsets of the circle.
	\item finally, we prove that the constructed pair $(C,K)$ is in $\mathfrak{D}$ by partitioning the circle line into different types of intervals and showing that the restriction of $M_K$ to each interval is the Lebesgue measure.
\end{enumerate}
(iii) The proof of the continuity of $\tau$ is defered to the next section, Proposition~\ref{p:continuitytau}.

\begin{proof}[Proof of Theorem~\ref{t:codingmap}]
(i) Let $\Gamma=(C,K)\in\mathfrak{D}$. For $\kappa\in\supp(K)$, we define the set-function $\mu_0^\kappa : \bigcup_{x\in V_C}\CC_x\to[0,1]$ by
\begin{equation}\label{e:mu0kappa}
	\mu_0^\kappa(\CS_x(y)):=\kappa(\comp_x(y)\cap\BS),\quad x,y\in V_C.
\end{equation}
We want to be in a position to apply Proposition~\ref{p:extmeasure} to each $\mu_0^\kappa$. Since $M_K=\lambda_\BS$, for $K$-almost every $\kappa$, for all $x\in V_C$, $\kappa(\partial x\cap\BS)=0$. Thus, for $K$-almost every $\kappa$ and all $x\in V_C$, $\sum_{A\in\CC_x}\mu_0^\kappa(A)=\kappa(\BS)=1$ and $\mu_0^\kappa(\CS_x(y))+\mu_0^\kappa(\CS_y(x))\geq 1$ for $y\neq x$. Furthermore, the function $x\mapsto\mu_0^\kappa(\CS_x(y))$ is order left-continuous on $V_C$ w.r.t.\ $\leq_y$ for all $y\in V_C$ for $K$-almost every $\kappa$. This can be shown by using \eqref{e:mu0kappa} and the continuity from above of the measure $\kappa$, together with the condition $M_K=\lambda_\BS$.

The tree $(V_C,c_V)$ is order separable. Indeed, the set $\Delta(C)\cup\nabla(C)$ is countable and one can construct a countable order dense set of $(V_C,c_V)$ by adding the straight line segments in $\|(C)$ that have at least one rational endpoint in $\BS\simeq[0,1]$. However, $(V_C,c_V)$ is not necessarily order complete. Since we defined $\blacksquare(C)$ to be the set of all closures of connected components of $C^\circ$, the endpoints of line segments are included in $(V_C,c_V)$. But we still need to add the leaves which are limits of an increasing sequence (w.r.t\ $\leq_y$ for some $y\in V_C$) of branch points. Thus we define $\overline{V}_C$ as $V_C$ together with an uncountable set of leaves given by the limit points of these increasing sequences, and we can easily extend $c_V$ to $\overline{c}_{V}$ so that the tree $(\overline{V}_C,\overline{c}_V)$ is order complete and order separable.

Furthermore, $(\overline{V}_C,\overline{c}_V)$ need not to be an algebraic continuum tree because $\edge(\overline{V}_C,\overline{c}_V)$ might be non-empty. Therefore, to apply Proposition~\ref{p:extmeasure} we also extend the tree $(\overline{V}_C,\overline{c}_V)$ to make it an order separable algebraic continuum tree. Let $\{v,w\}$ be an edge of $(\overline{V}_C,\overline{c}_V)$. Define $\widetilde{V}:=\overline{V}_C\uplus(\{(v,w)\}\times(0,1))$. We can extend $\overline{c}_V$ to a branch point map $\widetilde{c}$ on $\widetilde{V}$ in a canonical way such that $[v,w]=\{v,w\}\cup(\{(v,w)\}\times(0,1))$ (see Figure~\ref{f:fillingofedges}). $\overline{V}_C$ is not anymore order dense in $(\widetilde{V},\widetilde{c})$ so we also define, for $\kappa\in\supp(K)$, $\mu_0^\kappa$ on any $\CC_x$ with $x\in \{(v,w)\}\times(0,1)$. For $x\in\{(v,w)\}\times(0,1)$, let $\mu_0^\kappa(\{x\})=0$ and for $y\neq x$,
\begin{equation}
	\mu_0^\kappa(\CS_x(y)):=\begin{cases}
	\mu_0^\kappa(\CS_x(v)) & \text{if }v\in\CS_x(y),\\
	\mu_0^\kappa(\CS_x(w)) & \text{if }w\in\CS_x(y).
	\end{cases}
\end{equation}
Then for $K$-almost every $\kappa$, $\mu_0^\kappa$ still satisfies the conditions (1)-(3) of Proposition~\ref{p:extmeasure} on $(\widetilde{V},\widetilde{c})$. We repeat this extension for each edge in $\edge(\overline{V}_C,\overline{c}_V)$ which is countable as the set $\Delta(C)\cup\nabla(C)$ is itself countable. We denote by $(T,c)$ the extended tree. By construction, it is still order complete and order separable, and $\edge(T,c)=\emptyset$. Since $\mu_0^\kappa$ was extended to $\bigcup_{x\in T}\CC_x$ and $T$ is obviously dense in itself, we can now apply Proposition~\ref{p:extmeasure}. That is, for $K$-almost every $\kappa$, $\mu_0^\kappa$ has a unique extension to a probability measure $\mu^\kappa$ on $\CB(T,c)$.

Define $\nu\in \CM_1(\CM_1(T,c))$ as the pushforward of $K$ under the map $\kappa\mapsto\mu^\kappa$. The triplet $(T,c,\nu)$ will be the triplet $(T_\Gamma, c_\Gamma, \nu_Gamma)$ announced in the point (i) of Theorem \ref{t:codingmap}, but we do not write the subscripts $\Gamma$ for the sake of notation. Let us verify the a2m tree $(T,c,\nu)$ satisfies conditions (CM1)-(CM3). It is easy to see that (CM1) holds, and by \eqref{e:mu0kappa}, for all $x,y\in V_C$, and $\sigma>0$,
\begin{equation}
	\int\nu(\mathrm{d}\mu)e^{-\sigma\mu(\CS_x(y))}=\int K(\mathrm{d}\kappa)e^{-\sigma\mu^\kappa(\CS_x(y))}=\int K(\mathrm{d}\kappa)e^{-\sigma\kappa(\comp_x(y)\cap\BS)},
\end{equation}which is (CM2).
Moreover, points in $T\setminus V_C$ carry no atoms, so $\at(M_\nu)\subseteq V_C$. With (CM2), for each $x\in V_C$,
\begin{equation}
	M_{\nu}\{x\}=M_K(\comp_x(x)\cap\BS)=\lambda(\comp_x(x)\cap\BS),
\end{equation}
which is zero if $x\in\Delta(C)\cup\square(C)$ and strictly positive if $x\in\nabla(C)$, which yields (CM3). To see that it is unique (up to equivalence), notice that in the construction above, the extensions of edges of $(\overline{V}_C,\overline{c}_V)$ carry no mass and that the extension of $\mu_0^\kappa$ to $\mu^\kappa$ is unique for $K$-almost every $\kappa$ by Proposition~\ref{p:extmeasure}.

\begin{figure}[t]
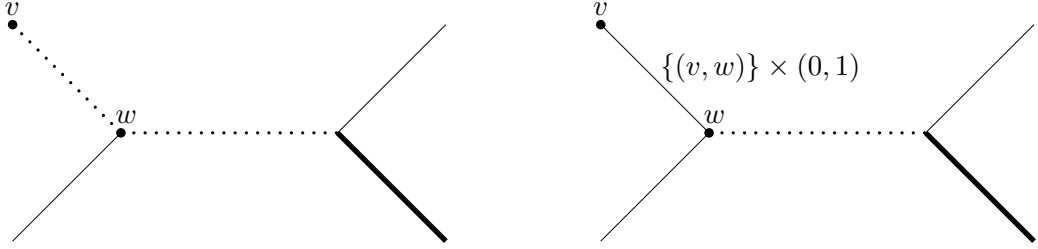

\begin{center}
\figurefillingofedges
\end{center}
\caption{The tree is not a continuum tree because $[v,w]=\{v,w\}$ is an edge. We extend the tree in a natural way so that $[v,w]$ is no longer an edge.}\label{f:fillingofedges}
\end{figure}

\vspace{.25cm}

(ii) We give the proof of the surjectivity in three steps. Let $\chi=(T,c,\nu)\in\T_2^{(2)}$. We construct a sub-triangulation $C$ and a two-level measure $K$ on $\BS$ such that $\tau(C,K)=\chi$. We can suppose w.l.o.g.\ that $(T,c)$ is order complete and that for all $v\in\br(T,c)$, $c_*(M_\nu^{\otimes 3})\{v\}>0$. To make the construction of $K$ easier, we also assume that for all $(v,w)\in\seg(T,c,\nu)$, $M_\nu$ restricted to $(v,w)$ is the Lebesgue measure (see Example~\ref{e:nobpnoat}).

\vspace{.25cm}

\noindent\textbf{Step 1: construction of $C$.} Fix $\rho\in\lf(T,c)$ and recall that $\rho$ induces a partial order relation $\leq_\rho$. We can extend this partial order to a total order $\leq$ in the following way. For $v\in\br(T,c)$, denote by $S_0(v)$, $S_1(v)$ and $S_2(v)$ the three components of $T\setminus\{v\}$ such that $S_0(v)=\CS_v(\rho)$ and with $S_1(v)$ and $S_2(v)$ chosen by picking an order for the two remaining components. Define now for all $v,w$,
\begin{equation}\label{e:totalorder}
		v\leq w \quad \Leftrightarrow \quad \big[v\leq_\rho w\big] \text{ or } \big[v\in S_1(c(v,w,\rho)) \text{ and }w\in S_2(c(v,w,\rho))\big].
\end{equation}
Using this total order on the tree, the construction of $C$ below can be understood as follows. Starting from the root $\rho$ which corresponds to 0 in the identification $\BS\simeq[0,1]$, we read through $T$ according to $\leq$. For each branch point $v$, we draw a triangle whose vertices on $\BS$ are given by the $M_\nu$-mass of points smaller than $v$ and the $M_\nu$-masses of the two components above $v$. For each leaf carrying an atom $w$ of $M_\nu$, we add a straight line segment to $C$ according to the $M_\nu$-mass of points smaller than $w$ and the $M_\nu$-mass of $w$ (see Figure~\ref{f:totalorder}).

\begin{figure}[t]
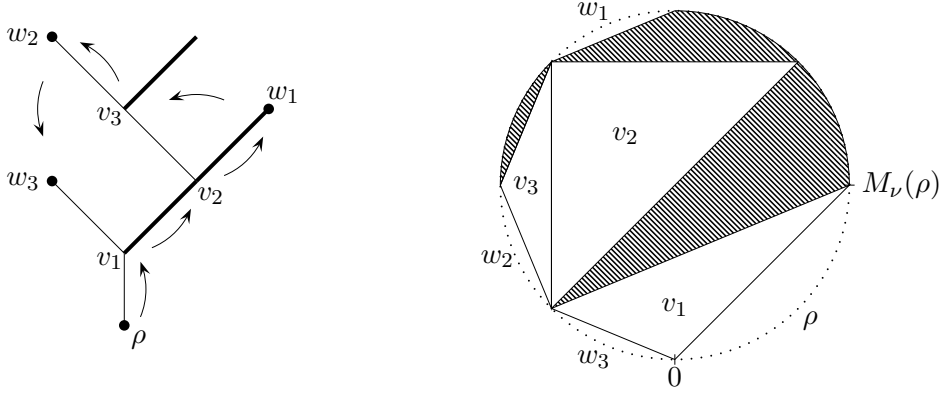

\begin{center}
\figuretotalorder
\end{center}
\caption{An algebraic tree $(T,c)$ and the corresponding sub-triangulation $C$ as constructed in Step 1 of the proof of Theorem~\ref{t:codingmap}(ii). $v_1,v_2,v_3$ are branch points and $\rho,w_1,w_2,w_3$ are leaves carrying an $M_\nu$-atom. More precisely, $M_\nu(\rho)=\frac{1}{4}$ and $M_\nu(w_1)=M_\nu(w_2)=M_\nu(w_3)=\frac{1}{8}$. We extend the partial order $\leq_\rho$ to a total order such that $\rho\leq v_1\leq v_2\leq w_1\leq v_3\leq w_2\leq w_3$.}\label{f:totalorder}
\end{figure}

We first introduce some more notations. For $a\in[0,1]$ and $b,c>0$ with $a+b+c\leq 1$, let $\Delta(a,b,c)\subseteq\D$ be the open triangle with vertices $a,a+b,a+b+c\in\BS \simeq[0,1]$, and let $\ell(a,b)\subseteq\D$ be the straight line from $a$ to $a+b$, and $L(a,b)$ the connected component of $\D\setminus\ell(a,b)$ containing $a+\frac{b}{2}\in\BS$.\\
Now for all $v\in\br(T,c)\cup\at(M_\nu)$, denote by
\begin{equation}
	\alpha(v):=M_\nu\left(\{u\in T:u<v\}\right)
\end{equation}
the total $M_\nu$-mass of points smaller than $v$ (with respect to $\leq$ defined in \eqref{e:totalorder}). In the sub-triangulation $C$ defined below, $\alpha(v)$ gives the first vertex of the triangle or circular segment corresponding to $v$. We define $C$ by
\begin{equation}\label{e:subtrisurj}
	\D\setminus C:=\biguplus_{v\in\br(T,c)}\Delta\big(\alpha(v),M_\nu(S_1(v)),M_\nu(S_2(v))\big)\uplus\biguplus_{w\in\at(M_\nu)}L\big(\alpha(w),M_\nu\{w\}\big).
\end{equation}
Let us show that $C$ is a sub-triangulation of the circle. By definition of $C$, $\mathrm{conv}(C)\setminus C$ is the disjoint union of open triangles, i.e.\ condition (Tri1) is satisfied. Furthermore, the extreme points of $\mathrm{conv}(C)$ are contained in $\BS$, and for $x, y, z \in\Delta(C)\cup\nabla(C)$ distinct, there are corresponding $u, v, w \in T$, and a triangle $c_{xyz}\in\Delta(C)$ corresponding to $c(u, v, w)$, which satisfies the requirements of (Tri2)’. Thus, by Proposition~\ref{p:characsubtriangulation}, $C\in\CT$. Note that $C$ is the sub-triangulation associated with the algebraic measure tree $(T,c,M_\nu)$ in the proof of Theorem~4.8 in \cite{LoehrWinter2021}.

To make the notations coincide with the ones from Lemma~\ref{l:inducedbpmap}, we will use the following correspondences in the rest of the proof. That is, a branch point $v\in\br(T,c)$ will correspond to the triangle $\overline{v}=\Delta\big(\alpha(v),M_\nu(S_1(v)),M_\nu(S_2(v))\big)$ in $\Delta(C)$ and an atom $w\in\at(M_\nu)$ to the circular segment $\overline{w}=L\big(\alpha(w),M_\nu\{w\}\big)$ in $\nabla(C)$. A leaf with zero mass $u\in\lf(T,c)$, but connected to the rest of the tree by a line segment carrying mass, will correspond to the midpoint $\overline{u}\in\square(C)$ of the corresponding ``filled'' circular segment. Finally, we associate, in a ``linear'' way, a point $z\in T$ belonging to a line segment with non-atomic mass in $\seg(T,c,\nu)$ to a straight line segment $\overline{z}\in\|(C)$. That is, if the segment $(v,w)\in\seg(T,c,\nu)$ corresponds to the circular segments $]x,y[$ and $]x',y'[$, the point $y=(1-t)v+tw$ with $t\in(0,1)$ is associated with the straight line segment $\overline{y}:=[(1-t)x+ty,(1-t)x'+ty']$ (see \eqref{e:partitionofb}). Denote by
\begin{equation}
	V_T\subseteq T
\end{equation}
the set of all points $x$ in the tree that have a corresponding subset $\overline{x}\in V_C$.

With these notations, $V_T$ is a subtree of $(T,c)$ and by construction, we have that for all $x,y\in V_T$,
\begin{equation}\label{e:Mnuislambda}
	M_\nu(\CS_x(y))=\lambda_\BS(\comp_{\overline{x}}(\overline{y})\cap\BS).
\end{equation}
Let $b\in\blacksquare(C)$ such that $b\cap\square(C)=\emptyset$, that is, $b$ corresponds to a segment line of the tree carrying mass and that is not connected to a leaf without mass. In this case, $b\cap\BS$ is the union of two disjoint intervals (see Figure~\ref{f:partitionfilledareas}, left). It is important for the next step of the proof to note that from the construction of $C$, one (and only one) of the two intervals is a singleton. Thus, the random $K$-mass we will assign to $b$ needs to be carried by the interval with non-empty interior to ensure that $M_K=\lambda_\BS$.

\vspace{.25cm}

\noindent\textbf{Step 2: construction of $K$. }Let $(x_{ij})_{(i,j)}$ be a random infinite matrix with distribution
\begin{equation}\label{e:infmeasure}
	\int\nu^{\otimes \infty}(\du{\mu})\int\bigotimes_{i\geq 1}\mu_i^{\otimes \infty}(\cdot).
\end{equation}
We assumed the tree $(T,c)$ to be order separable and order complete. Thus the component topology on $(T,c)$ is Polish (see Remark~\ref{r:conditionmetrizable}) and we know from \cite[Proposition~3.7]{Meizis2019} that, almost surely,
\begin{enumerate}
	\item[(i)] for every $i\in\N$, the weak limit $\mu_i:=\underset{n\to\infty}{\text{w-lim}}\frac{1}{n}\sum_{j=1}^n\delta_{x_{ij}}$ exists and has law $\nu$,
	\item[(ii)] the two-level measure $\frac{1}{m}\sum_{i=1}^m \delta_{\mu_i}$ converges weakly to $\nu$, and
	\item[(iii)] for every $j\in\N$, the sequence $(x_{ij})_i$ is dense in $\supp(M_\nu)$.
\end{enumerate}

Fix $i,n\in\N$. We write $\mu_i^n:=\frac{1}{n}\sum_{j=1}^n\delta_{x_{ij}}$ and we first define $\kappa_i^n$ a probability measure on the circle line corresponding to $\mu_i^n$. For each $x_{ij}$, we put weight $\frac{1}{n}$ on a subset of the circle depending on the position of $x_{ij}$ in $(T,c)$. We will make use of the correspondence between points of $V_T\subseteq T$ and subsets of $\D$ discussed at the end of Step~1.  Since $(T,c,\nu)\in\T_2^{(2)}$, $\at(\mu)\subseteq\lf(T,c)$ for all $\mu\in\supp(\nu)$ so we have the following possible cases (see Figure~\ref{f:typesofboundaries}) that correspond to a partition of $T$ in five sets $A_1,\cdots A_5$:

\begin{figure}[t]
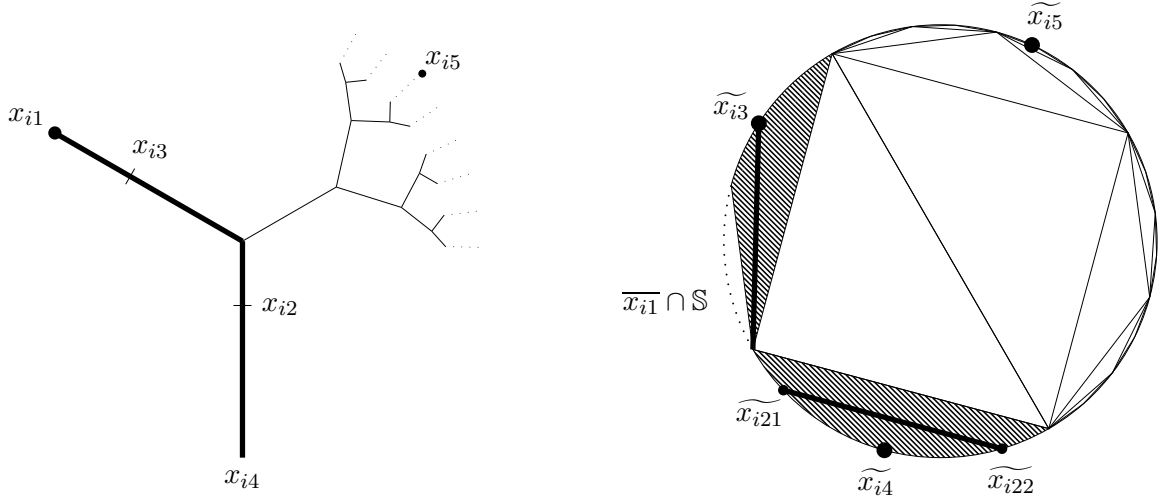

\begin{center}
\figuretypesofboundaries
\end{center}
\caption{For each atom $x_{i1},...,x_{i5}$ of $\mu_i^n$, $\kappa_i^n$ assigns weight $\frac{1}{n}$ on a subset of the circle, which depends on the position of the atom in the tree. For the $M_\nu$-atom leaf $x_{i1}$, $\kappa_i^n$ assigns the rescaled Lebesgue measure on the corresponding circular segment $\overline{x_{i1}}\cap\BS$. For $x_{i2}$, $\kappa_i^n$ splits the mass $\frac{1}{n}$ between the endpoints $\widetilde{x_{i21}}$ and $\widetilde{x_{i22}}$ of the segment $\overline{x_{i2}}$. For the three other ones, $\widetilde{x_k}$ carries weight $\frac{1}{n}$.}\label{f:typesofboundaries}
\end{figure}

\begin{enumerate}
	\item Set $A_1$: if $x_{ij}\in\at(M_\nu)$, $\overline{x_{ij}}$ is the circular segment $L\big(\alpha(x_{ij}),M_\nu\{x_{ij}\})\big)$ and we will put the rescaled Lebesgue measure on the arc $\overline{x_{ij}}\cap\BS$.
	\item Suppose now that $x_{ij}$ belongs to a line segment with non-atomic mass, or equivalently that $\deg(x_{ij})=2$. Since $(T,c)$ is order complete, there exist $v,w\in T$ such that $(v,w)\in\seg(T,c,\nu)$ and $x_{ij}\in(v,w)$.
	\begin{enumerate}
		\item Set $A_2$: if $(v,w)$ is adjacent to a leaf, i.e.\ $\overline{x_{ij}}$ is included in a filled circular segment, then we put weight $\frac{1}{2n}$ on the endpoints $\widetilde{x_{ij1}}$ and $\widetilde{x_{ij2}}$ of the line segment $\overline{x_{ij}}$.
		\item Set $A_3$: if not, then $\overline{x_{ij}}\subset b$ where $b\in\blacksquare(C)$ is such that $b\cap\BS$ is the union of an interval and a singleton. In this case, we put weight $\frac{1}{n}$ on the endpoint $\widetilde{x_{ij}}$ of the line segment $\overline{x_{ij}}$ that is not the singleton.
	\end{enumerate}
	\item Finally, suppose that $x_{ij}\in\lf(T,c)\setminus\at(M_\nu)$.
	\begin{enumerate}
		\item Set $A_4$: if $x_{ij}$ is adjacent to a line segment with non-atomic mass, then $\overline{x_{ij}}\in\square(C)$ and we put weight $\frac{1}{n}$ on $\widetilde{x_{ij}}=\overline{x_{ij}}$.
		\item Set $A_5$: if not, $x_{ij}\notin V_T$, i.e.\ $\overline{x_{ij}}$ is not defined. However, it naturally corresponds to a specific point of the circle. To see this, let $y\in\br(T,c)$. Then there exists $(z_n)_{n\in\N}$ a sequence of distinct branch points in $T$ such that $z_n\leq_y z_{n+1}\leq_y x_{ij}$ for all $n\in\N$. For all $n\in\N$, let $a_n,b_n\in\BS\simeq[0,1]$ such that $\comp_{\overline{z_n}}(\overline{z_{n+1}})\cap\BS=[a_n,b_n]$. We have $[a_{n+1},b_{n+1}]\subseteq[a_n,b_n]$. By \eqref{e:Mnuislambda}, $\lambda_\BS([a_n,b_n])=M_\nu(\CS_{z_n}(z_{n+1}))=M_\nu(\CS_{z_n}(x_{ij}))$. But since $x_{ij}\notin\at(M_\nu)$ and $x_{ij}$ is not adjacent to a line segment with non-atomic mass, $\lim_{n\to\infty}M_\nu(\CS_{z_n}(x_{ij}))=0$. Therefore, there exists $\widetilde{x_{ij}}\in\BS$ unique such that $\widetilde{x_{ij}}=\bigcap_{n\in\N}[a_n,b_n]$. We put weight $\frac{1}{n}$ on $\widetilde{x_{ij}}$.
	\end{enumerate}
\end{enumerate}

Define the random probability measure $\kappa_i^n$ on $\BS$ by
\begin{equation}
	\kappa_i^n:=\sum_{x_{ij}\in A_1}\frac{\lambda_\BS(\cdot\cap \overline{x_{ij}})}{n\lambda_\BS(\BS\cap\overline{x_{ij}})}+\sum_{x_{ij}\in A_2}\frac{\delta_{\widetilde{x_{ij1}}}(\cdot)+\delta_{\widetilde{x_{ij2}}}(\cdot)}{2n}+\sum_{x_{ij}\in A_3\cup A_4\cup A_5}\frac{\delta_{\widetilde{x_{ij}}}(\cdot)}{n}.
\end{equation}

Now for each $i,n\in\N$, $\kappa_i^n$ is a random variable in $\CM_1(\BS)$ which is compact when equipped with the weak topology. Therefore, the sequence $(\kappa_i^n)_n$ is tight and has a subsequence $(\kappa_i^{\varphi_i(n)})_n$ that converges to some random measure $\kappa_i$, for each $i\in\N$. Before taking the weak limit of the empirical distribution of the sequence $(\kappa_i)_i$, note first that a.s.\ $\mu_i$ and $\kappa_i$ agree on (union of) corresponding components on the tree $T$ and on the circle line $\BS$. Indeed, by construction, a.s.\ for all $x,y\in V_T$,
\begin{equation}\label{e:kappainmuin}
	\kappa_i^n(\comp_{\overline{x}}(\overline{y})\cap\BS)=\mu_i^n(\CS_x(y)).
\end{equation}
Furthermore, consider $x,y,x',y'\in V_C$. Then $\comp_x(y)$ and $\comp_{x'}(y')$ are two circular segments delimited by two straight line segments that are disjoint or equal. Thus $\comp_x(y)\cup\comp_{x'}(y')$ is either disjoint, the whole circle $\BS$ or one component if one is included in the other (see Figure~\ref{f:unionofcomponents}). Therefore, we have a.s.\ for all $x_k,y_k\in V_T$,
\begin{equation}
	\kappa_i^n\left(\bigcup_{k\in\N}\comp_{\overline{x_k}}(\overline{y_k})\cap\BS\right)=\mu_i^n\left(\bigcup_{k\in\N}\CS_{x_k}(y_k)\right).
\end{equation}

\begin{figure}[t]
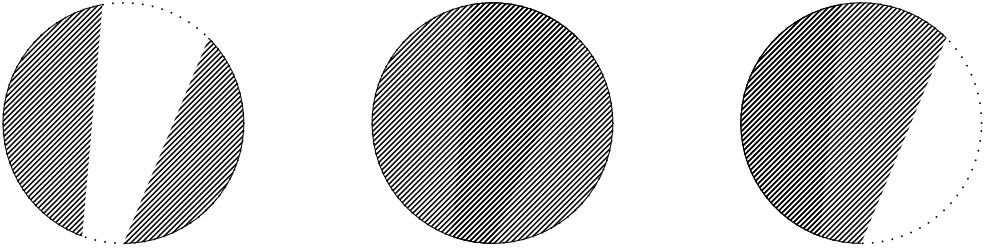

\begin{center}
\figureunionofcomponents
\end{center}
\caption{Thus union of two components is either disjoint (\emph{left}), the whole circle $\BS$ (\emph{middle}) or one component if one is included in the other (\emph{right}).}\label{f:unionofcomponents}
\end{figure}

Then we can generalize it by saying that for $n$ large enough, a.s. for all $x_k,y_k\in V_C$,
\begin{equation}
	(\kappa_1^n,...,\kappa_m^n)\left(\bigcup_{k\in\N}\comp_{\overline{x_k}}(\overline{y_k})\cap\BS\right)=(\mu_1^n,...,\mu_m^n)\left(\bigcup_{k\in\N}\CS_{x_k}(y_k)\right).
\end{equation}
Taking the limit of the subsequences $(\kappa_i^{\varphi_i(n)})_n$ when $n$ goes to infinity for $1\leq i\leq m$,
\begin{equation}
	(\kappa_1,...,\kappa_m)\left(\bigcup_{k\in\N}\comp_{\overline{x_k}}(\overline{y_k})\cap\BS\right) \quad \text{and} \quad (\mu_1,...,\mu_m)\left(\bigcup_{k\in\N}\CS_{x_k}(y_k)\right).
\end{equation}
have same distribution for all $x_k,y_k\in V_T$.

We now denote by $K_m$ the empirical distribution of the sequence $(\kappa_i)_i$, that is,
\begin{equation}
	K_m:=\frac{1}{m}\sum_{i=1}^m\delta_{\kappa_i},
\end{equation}
which is a random variable in $\CM_1(\CM_1(\BS))$. Since $\CM_1(\CM_1(\BS))$ is also compact, the sequence $(K_m)_m$ is tight, so that there exists a subsequence $(K_{\varphi(m)})_m$ that converges to some random two-level measure $K$. For all $m\in\N$, $x_k,y_k\in V_T$, and $\sigma\geq0$,
\begin{equation}
	\frac{1}{m}\sum_{i=1}^m e^{-\sigma\kappa_i\left(\bigcup_{k\in\N}\comp_{\overline{x_k}}(\overline{y_k})\cap\BS\right)} \quad \text{and} \quad \frac{1}{m}\sum_{i=1}^m e^{-\sigma\mu_i\left(\bigcup_{k\in\N}\CS_{x_k}(y_k)\right)}
\end{equation}
have same distribution. Recall that the two-level measure $\frac{1}{m}\sum_{i=1}^m \delta_{\mu_i}$ converges weakly to $\nu$. Thus, taking the limit of the subsequence $(K_{\varphi(m)})_m$, for all $x_k,y_k\in V_T$, and $\sigma\geq0$,
\begin{equation}
	\int K(\mathrm{d}\kappa)e^{-\sigma\kappa\left(\bigcup_{k\in\N}\comp_{\overline{x_k}}(\overline{y_k})\cap\BS\right)} \quad \text{and} \quad \int \nu(\mathrm{d}\mu)e^{-\sigma\mu\left(\bigcup_{k\in\N}\CS_{x_k}(y_k)\right)}
\end{equation}
have same distribution. But the term on the right above is deterministic, so we have that a.s. for all $x,y\in V_T$, and $\sigma\geq0$,
\begin{equation}\label{e:CM2inproof}
	\int K(\mathrm{d}\kappa)e^{-\sigma\kappa\left(\bigcup_{k\in\N}\comp_{\overline{x_k}}(\overline{y_k})\cap\BS\right)} = \int \nu(\mathrm{d}\mu)e^{-\sigma\mu\left(\bigcup_{k\in\N}\CS_{x_k}(y_k)\right)},
\end{equation}
which yields in particular (CM2).

\noindent\textbf{Step 3: $M_K=\lambda_\BS$ a.s.} By taking the derivative of \eqref{e:CM2inproof} at $\sigma=0$ and with \eqref{e:Mnuislambda}, a.s.\ for all $x_k,y_k\in V_C$,
\begin{equation}
	M_K\left(\bigcup_{k\in\N}\comp_{\overline{x_k}}(\overline{y_k})\cap\BS\right)=M_\nu\left(\bigcup_{k\in\N}\CS_{x_k}(y_k)\right)=\lambda_\BS\left(\bigcup_{k\in\N}\comp_{\overline{x_k}}(\overline{y_k})\cap\BS\right).
\end{equation}
In other words, $M_K=\lambda_\BS$ on the $\sigma$-algebra generated by $\{\comp_{\overline{x}}(\overline{y})\cap\BS:x,y\in V_T\}$. To prove the equality on the Borel $\sigma$-algebra of $\BS$, we partition the circle line $\BS$ in the following way. Since $\Delta(C)$ is countable, there exist countably many disjoint intervals of non-empty interior $I_p$, $p\in\N$, such that
\begin{equation}
	\BS=\biguplus_{p\in\N}I_p,
\end{equation}
and for each $p\in\N$, $I_p$ satisfies one of the following (see Figure~\ref{f:typesofboundaries} for a sub-triangulation with the different types of boundaries):
\begin{enumerate}
	\item[(a)] $I_p$ corresponds to an atom leaf of $(T,c)$, that is, there exists $w\in\at(M_\nu)$ such that $(I_p)^\circ=\overline{w}\cap\BS$,
	\item[(b)] $I_p$ corresponds to a line segment in $\seg(T,c,\nu)$, i.e.\ there exists $b\in\blacksquare(C)$ such that $(I_p)^\circ=(b\cap\BS)^\circ$,
	\item[(c)] $I_p\subseteq C$ and for all distinct $x,y\in I_p$, there exists $c_{xy}\in\Delta(C)$ such that $x,y$ belong to different connected components of $\D\setminus\partial c_{xy}$.
\end{enumerate}
We show that a.s.\ $M_K=\lambda_\BS$ by showing that a.s.\ for any $I:=I_p$, $M_{K|I}=\lambda_I$.
\begin{enumerate}
	\item[(a)] Suppose $I^\circ=\overline{w}\cap\BS$ for $w\in\at(M_\nu)$. For all $n\leq i$, we defined $\kappa_i^n$ such that
	\begin{equation}
		\kappa^n_{i|I}=C^{i,n}_I\lambda_I
	\end{equation}
	where $C^{i,n}_I$ is a real-valued random-variable. Taking the weak limit of $(\kappa_i^{\varphi_i(n)})_n$ when $n$ goes to infinity,
	\begin{equation}
		\kappa_{i|I}=C^i_I\lambda_I
	\end{equation}
	a.s.\ for some real-valued random-variable $C^i_I$. Therefore, a.s.\ for all $m\in\N$,
	\begin{equation}
		M_{K_m|I}=\left(\frac{1}{m}\sum_{i=1}^mC^i_I\right)\lambda_I.
	\end{equation}
	Taking the weak limit of $(K_{\varphi(m)})_m$ when $m$ goes to infinity,
	\begin{equation}
		M_{K|I}=C_I\lambda_I
	\end{equation}
	a.s.\ for some real-valued random-variable $C_I$. Noticing that $I$ belongs to the $\sigma$-algebra generated by $\{\comp_{\overline{x}}(\overline{y})\cap\BS:x,y\in V_C\}$, we have $M_K(I)=\lambda_I(I)$ so that $C_I=1$ a.s.
	\item[(b)] Suppose that $I^\circ=(b\cap\BS)^\circ$ for some $b\in\blacksquare(C)$. We will use here the assumption that for the corresponding $(v,w)\in\seg(T,c,\nu)$, $M_\nu$ restricted to $(v,w)$ is the Lebesgue measure.
	\begin{itemize}
		\item If $b\cap\BS$ is the union of two disjoint intervals (one of which is a singleton), then the map that associates a point in $(v,w)$ to the endpoint of the corresponding segment in $b$ is linear. Therefore, with the assumption on $\nu$, $M_{K|I}=\lambda_I$.
		\item If not, let $x,x'\in\BS$ such that $b\cap\BS=[x,x']$, $y=y'=\frac{x+x'}{2}$ (see Figure~\ref{f:partitionfilledareas}). In this case, $(v,w)$ is adjacent to a leaf and we can assume w.l.o.g.\ that $w\in\lf(T,c)$. The point $u=(1-t)v+tw$, $t\in(0,1)$ corresponds to the segment with endpoints $\widetilde{u_1}=(1-t)x+ty$ and $\widetilde{u_2}=(1-t)x'+ty'$ (see Figure~\ref{f:typesofboundaries}). Therefore, by construction and assumption on $\nu$, $M_{K|(x,y)}$ is proportional to $\lambda_{(x,y)}$ and $M_{K}(x,y)=\frac{M_{K}(x,x')}{2}=\frac{\lambda_\BS(x,x')}{2}=\lambda_\BS(x,y)$ because $[x,x']$ is the intersection of a component with $\BS$. Thus $M_{K|(x,y)}=\lambda_{(x,y)}$ and since it also holds for $(y',x')$, $M_{K|I}=\lambda_I$.
	\end{itemize}
	\item[(c)] Finally, we assume that $I\subseteq C$ and for all distinct $x,y\in I$, there exists $c_{xy}\in\Delta(C)$ such that $x,y$ belong to different connected components of $\D\setminus\partial c_{xy}$. Note that in this case, $[x,y]\cap c_{xy}\neq\emptyset$. In other words, the set of endpoints of (boundaries of) triangles in $\Delta(C)$ is dense in $I$. Let $J=[a,b]\subseteq I$. There exist two sequences $(a_n)_n$ and $(b_n)_n$ in $I$ such that for all $n$, $a_n\in\partial_\D v$ and $b_n\in\partial_\D w$ for some $v,w\in\Delta(C)$ and $a_n\uparrow a$ and $b_n\downarrow b$. Since $a_n\in\partial_\D v$ and $b_n\in\partial_\D w$, $[a_n,b_n]$ is the union of countably many disjoint components of $C$ intersected with $\BS$, so that $M_K([a_n,b_n])=\lambda_\BS([a_n,b_n])$. Therefore, $M_K(J)=\lim_{n\to\infty}M_K([a_n,b_n])=\lambda_\BS(J)$. Since it holds for all interval $J\subseteq I$, $M_{K|I}=\lambda_I$.
\end{enumerate}

We have shown that a.s.\ the random two-level measure $K$ together with the sub-triangulation $C$ is such that $\tau(C,K)=(T,c,\nu)$ and $M_K=\lambda_\BS$. Therefore, by taking a realization of $K$ such that it holds, we have shown surjectivity of $\tau$.
\end{proof}

\section{Topologies on the subspace of binary algebraic two-level measure trees}
\label{S:topologies}

\josue{Recall that in Section~\ref{s:m2mspaces} we have equiped $\T^{(2)}$ with the Gromov-weak topology using a metric representation that comes from the branch point distribution.} Here we introduce on the space of binary algebraic two-level measure trees another notion of convergence called \emph{two-level sample shape convergence}. It exploits the idea of the Gromov-weak topology to sample finite sub-spaces and then require these to converge in distribution. Whereas one samples metric sub-spaces in the Gromov-weak topology, we now consider subtrees of a2m trees as combinatorial objects, which will make it useful to show convergence of tree-valued Markov chains in the future. We show that this topology is equivalent to the two-level bpdd-Gromov-weak topology on $\T_2^{(2)}$, and that both topologies are compact.

\subsection{Two-level sample shape convergence}

Obviously the two-level sample shape convergence is similar to the notion of sample shape convergence defined on the space of algebraic measure trees (see \cite{LoehrWinter2021}, that we follow closely for this generalization to the two-level case). A sequence of trees converges to a limit tree if all random tree shapes spanned by finite samples converge weakly to the corresponding limit shapes. On $\T_2^{(2)}$, there is a two-level sampling procedure: we first sample a finite number of measures on the tree from the two-level measure and then sample a finite number of points according to each sampled measure. Therefore, the notion of tree shape that we define keeps track of this two-level sampling by using double indices.

For $m\in\N$ and $\vn\in\N^m$, we denote by
\begin{equation}
	\N_{m,\vn}=\{(i,j):1\leq i \leq m,1\leq j\leq n_i\}
\end{equation}
the set of indices. The sampled subtrees from an a2m tree will be cladograms, which are binary finite trees with labelled leaves. Since the two-level measure on the tree may have atoms on leaves, a given leaf may be sampled several times. We therefore need to allow multi-labels in cladograms.

\begin{definition}[$(m,\vn)$-cladogram]
For $m\in\N$ and $\vn\in\N^m$, an \emph{$(m,\vn)$-labelled cladogram} is a binary, finite algebraic tree $(C, c)$ consisting only of leaves and branch points together with a surjective labelling map $\zeta\colon \N_{m,\vn} \to \mathrm{lf}(C)$. An \emph{$(m,\vn)$-cladogram} $(C, c, \zeta)$ is an $(m,\vn)$-labelled cladogram such that $\zeta$ is also injective.

We call two $(m,\vn)$-labelled cladograms $(C_1, c_1, \zeta_1)$ and $(C_2, c_2, \zeta_2)$ \emph{isomorphic} if there exists a tree isomorphism $\phi$ from $(C_1, c_1)$ onto $(C_2, c_2)$ such that $\zeta_2 = \phi \circ\zeta_1$. We then write
\begin{equation}
	\overline{\FC}_{m,\vn} := \{ \text{isomorphism classes of }(m,\vn)\text{-labelled cladograms}\}
\end{equation}
and
\begin{equation}
\FC_{m,\vn} := \{(C, c, \zeta) \in \overline{\FC}_{m,\vn} : \zeta \text{ injective}\}.
\end{equation}
\end{definition}

The shape function encodes as cladograms the shape of the subtree spanned by finite samples of points (see Figure~\ref{f:TLshapefunction}).

\begin{definition}[Shape function]
Let $(T, c)$ be a binary algebraic tree, $m\in\N$, $\vn\in\N^m$, and $u_{ij}\in T \setminus \mathrm{br}(T)$ for $(i,j)\in\N_{m, \vn}$. Then there exists a unique (up to isomorphism) $(m,n)$-labelled cladogram
\begin{equation}
	\Fs_{(T,c)}(\mat{u}) = (C, c_C , \zeta)
\end{equation}
with $\mathrm{lf}(C) = \{u_{ij}\}_{(i,j)\in\N_{m,\vn}}$ and $\zeta(i,j) = u_{ij}$, such that the identity on $\mathrm{lf}(C)$ extends to a tree homomorphism $\pi$ from $C$ onto $c\left((\{u_{ij}\}_{(i,j)\in\N_{m,\vn}})^3\right)$, i.e.\ for all $(i_1,j_1),(i_2,j_2),(i_3,j_3)\in\N_{m,\vn}$,
\begin{equation}
	\pi(c_C (u_{i_1j_1}, u_{i_2j_2}, u_{i_3j_3})) = c(u_{i_1j_1}, u_{i_2j_2}, u_{i_3j_3}).
\end{equation}
We will refer to $\Fs_{(T,c)} (\mat{u}) \in \overline{\FC}_m$ as the \emph{shape} of $\mat{u}$ in $(T, c)$.
\end{definition}

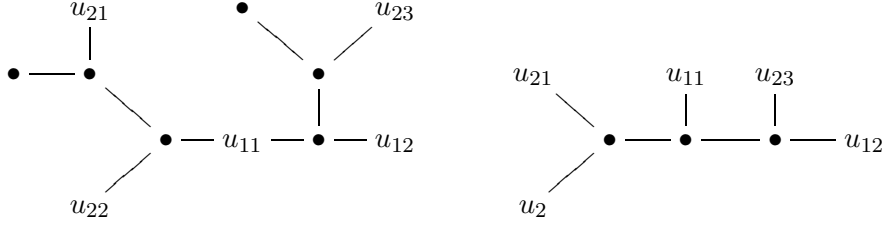
\begin{figure}[t]%-----------shape
\[
\xymatrix@=1pc{
&u_{21}\ar@{-}[d]&&{\bullet}&&u_{23}   &   &&&\\
{\bullet}\ar@{-}[r]&{\bullet}\ar@{-}[dr]&&&{\bullet}\ar@{-}[ul]\ar@{-}[ur]&   &&   u_{21}\ar@{-}[dr]&&u_{11}\ar@{-}[d]&u_{23}\\
&&{\bullet}\ar@{-}[r]&u_{11}\ar@{-}[r]&{\bullet}\ar@{-}[r]\ar@{-}[u]&u_{12}   &&   &{\bullet}\ar@{-}[r]&{\bullet}\ar@{-}[r]&{\bullet}\ar@{-}[r]\ar@{-}[u]&u_{12}\\
&u_{22}\ar@{-}[ur]&&&&   &&   u_2\ar@{-}[ur]&&&}\]
\caption{A tree $T$ and the shape $\Fs_{(T,c)}(u_{11},u_{12},u_{21},u_{22},u_{23})$.}
\label{f:TLshapefunction}
\end{figure}

We introduce a notion of convergence on $\T_2^{(2)}$ based on the weak convergence of random tree shapes spanned by finite samples. This topology is defined as the topology induced by the two-level shape polynomials, which are test functions evaluating the tree shape distributions.

\begin{definition}[Two-level shape polynomials]
A \emph{two-level shape polynomial} is a function $\Phi^{m,\vn,\varphi} \colon \T^{(2)}_2 \to \R$ of the form
\begin{equation}
	\Phi^{m,\vn,\varphi}(\chi):=\int_{(\CM_1(T))^m}\nu^{\otimes m}(\du{\mu})\int_{T^{|\vn|}}\bigotimes_{i=1}^m\mu_i(\du{u_i})\varphi\left(\Fs_{(T,c)} (\mat{u})\right),
\end{equation}
where $\chi=(T,c,\nu)$, $m\in\N$, $\vn\in\N^m$ and $\varphi\colon\FC_{m,\vn}\to\R$. We write $\Pi_\Fs^{(2)}$ for the set of all two-level shape polynomials.
\end{definition}

\begin{definition}[Two-level sample shape topology]
The \emph{two-level sample shape topology} on $\T_2^{(2)}$ is defined as the initial topology induced by $\Pi_\Fs^{(2)}$.
\end{definition}

\begin{remark}[Tree shape distribution]\label{r:treeshapedistribution}
Fix $m\in\N$ and $\vn\in\N^m$. For all $\chi=(T,c,\nu)$, we define the \emph{$(m,\vn)$-tree shape distribution} $\mathfrak{S}_{m,\vn}(\chi)$ as the probability measure on $\FC_{m,\vn}$ such that for all $\varphi\colon\FC_{m,\vn}\to\R$,
\begin{equation}
	\int_{\FC_{m,\vn}}\mathrm{d}\mathfrak{S}_{m,\vn}(\chi)~\varphi=\int_{\CM_1(T)}\nu^{\otimes m}(\du{\mu})\int_{T^{|\vn|}}\bigotimes_{i=1}^m\mu_i(\du{u_i})\varphi\left(\Fs_{(T,c)} (\mat{u})\right).
\end{equation}
Then the two-level sample shape topology is induced by the set of functions
\josue{\begin{equation}
	\mathcal{F}_\Fs:=\{\mathfrak{S}_{m,\vn}:m\in\N,\vn\in\N\}.
\end{equation}}
\end{remark}

The following result gives that on $\T_2^{(2)}$, two-level sample shape convergence implies two-level bpdd-Gromov-weak convergence.

\begin{proposition}\label{p:strongertopology}
On $\T_2^{(2)}$, the two-level sample shape topology is stronger than the two-level bpdd-Gromov-weak topology.
\end{proposition}

\begin{proof}
By definition, the two-level bpdd-Gromov-weak topology is induced by the set $\Pi_\iota^{(2)}$ of polynomials of the form
\begin{equation}
	\Phi(\chi):=\int\nu^{\otimes m}(\du{\mu})\int\bigotimes_{i=1}^m\mu_i(\du{u_i})\varphi\left(\big(r_\lambda(u_{ij},u_{i'j'})_{(i,j),(i',j')}\big)\right),
\end{equation}
where $\chi=(T,c,\nu)$, $\lambda:=c_{*}M_\nu^{\otimes 3}$, $m\in\N$, $\vn\in\N^m$ and $\varphi\in\CC_b(\R^{|\vn|^2})$ (see Definition~\ref{d:TLgromovweak} and Definition~\ref{d:TLbpdd}). Now, the set of $\phi\in\CC_b(\R^{|\vn|^2})$ that are Lipschitz continuous is convergence determining for probability measure on $\R^{|\vn|^2}$. Hence, the subset of $\Psi\in\Pi_\iota^{(2)}$ with
\begin{equation}
	\Psi(T,c,\nu)=\int\nu^{\otimes m}(\du{\mu})\int\bigotimes_{k=1}^m\mu_k^{\otimes n_k}(\du{u_k})\phi\left(\big(r_\lambda(u_{ij},u_{i'j'})_{(i,j),(i',j')}\big)\right)
\end{equation}
for some $m\in\N$, $\vn\in\N^m$ and a Lipschitz continuous function $\phi\in\CC_b(\R^{|\vn|^2})$ also induces the two-level bpdd-Gromov-weak topology. Therefore, it is enough to show that such a $\Psi$ is continuous on $\T^{(2)}_2$ with respect to the sample shape topology. To do this, we show that the restriction of $\Psi$ to $\T^{(2)}_2$ can be uniformly approximated by polynomials in $\Pi_\Fs^{(2)}$.

For $p\in\N$ with $3p\geq m$, we define
\begin{equation}
	\Phi_p(T,c,\nu)=\int\nu^{\otimes 3p}(\du{\mu})\int\bigotimes_{k=1}^m\mu_k^{\otimes n_k}(\du{u_k})\bigotimes_{k=m+1}^{3p}\mu_k(\mathrm{d}u_{k1})\phi\left(\big(r_{\lambda_{p,\mau}}(u_{ij},u_{i'j'})_{(i,j),(i',j')}\big)\right)
\end{equation}
where
\begin{equation}
	\lambda_{p,\mau}:=\frac{1}{p}\sum_{l=0}^p\delta_{c(u_{3l+1,1},u_{3l+2,1},u_{3l+3,1})}.
\end{equation}
Note that whether or not $c(u_{3l+1,1},u_{3l+2,1},u_{3l+3,1})$ lies on $[u_{(i,j)}, u_{(i',j')}]$, for some $l\in\{0, ..., p-1\}$ and $(i,j),(i',j')\in\N_{m,\vn}$ only depends on the shape $\Fs_{(T,c)}(\mau)$ and thus the restriction of $\Phi_p$ to $\T^{(2)}_2$ belongs to $\Pi_\Fs^{(2)}$.

To show that $\Phi_p$ approximates $\Psi$, we uniformly bound the distance of the empirical branch point distribution to the branch point distribution. Recall that for $x, y \in T$,
\begin{equation}
	r_\lambda(x,y) := \lambda([x,y]) -\frac{1}{2} \lambda(\{x\}) -\frac{1}{2}\lambda(\{y\}).
\end{equation}
Thus, denoting by $L$ the Lipschitz constant of $\phi$ w.r.t.\ the $l_\infty$-norm on $\R^{|\vn|^2}$,
\begin{equation}
\begin{aligned}
	||\Psi-\Phi_p||_\infty & \leq \sup_{(T,c,\nu)\in\T_2^{(2)}}\int\nu^{\otimes 3p}(\du{\mu})\int\bigotimes_{k=1}^m\mu_k^{\otimes n_k}(\du{u_k})\bigotimes_{k=m+1}^{3p}\mu_k(\mathrm{d}u_{k1})L.2\sup_{I\in\CI_T}|\lambda(I)-\lambda_{p,\mau}(I)|\\
	& \leq 2L\sup_{(T,c,\nu)\in\T_2^{(2)}}\int(M_\nu)^{\otimes 3p}(\mathrm{d}\mau)\sup_{I\in\CI_T}|\lambda(I)-\lambda_{p,\mau}(I)|,
\end{aligned}
\end{equation}
where $\CI_T:=\{[x,y]:x,y\in T\}$ is the collection of all intervals of the algebraic tree $T$.

Then \cite[Lemma~A.4]{LoehrWinter2021} gives the following estimates of the rate of convergence in the approximation of the branch point distribution by empirical distribution:
\begin{equation}
	\int(M_\nu)^{\otimes 3p}(\mathrm{d}\mau)\sup_{I\in\CI_T}|\lambda(I)-\lambda_{p,\mau}(I)|\leq 96\sqrt{\frac{\mathrm{dim_{\mathrm{VC}}(\CI_T)}}{p}}=96\sqrt{\frac{2}{p}},
\end{equation}
where $dim_{\mathrm{VC}}(\CI_T)$ is the Vapnik-Chervonenkis dimension of $\CI_T$ which can easily be shown to be 2 (see \cite[Example~A.2]{LoehrWinter2021} for more details). This concludes the proof.
\end{proof}

\begin{corollary}[Metrizability]\label{c:metrizability}
The two-level sample shape topology is metrizable on $\T_2^{(2)}$.
\end{corollary}

\begin{proof}
By Proposition~\ref{p:strongertopology}, the two-level sample shape topology is stronger than the two-level bpdd-Gromov-weak topology, which is Hausdorff by Proposition~\ref{p:separablemetrizable}. Hence the two-level sample shape topology is also Hausdorff. Moreover, by Remark~\ref{r:treeshapedistribution}, it is induced by the set $\mathcal{F}_\Fs$, which is a countable family of functions with values in metrizable spaces. Thus we can define a pseudo-metric on $\T_2^{(2)}$ that induces the two-level sample shape topology as follows: for $\chi,\chi'\in\T_2^{(2)}$,
\begin{equation}
	d_\Fs(\chi,\chi'):=\sum_{m\in\N}\frac{1}{2^m}\sum_{\vn\in\N^m}\frac{1}{2^{|\vn|}}\min\Big(d_{\mathrm{Pr}}\big(\mathfrak{S}_{m,\vn}(\chi),\mathfrak{S}_{m,\vn}(\chi')\big),1\Big),
\end{equation}
where $d_{\mathrm{Pr}}$ denotes the Prokhorov distance. Therefore, the two-level sample shape topology is Hausdorff and pseudo-metrizable. In particular, if $d_\Fs(\chi,\chi')=0$ then every open set that contains $\chi$ also contains $\chi'$, so that $d_\Fs$ is actually a metric.
\end{proof}

\subsection{Equivalence and compactness of topologies}

In this section, we first show that the coding map $\tau$ is continuous when $\T_2^{(2)}$ is equipped with the two-level sample shape topology. This implies that it is a compact topology. Finally, we prove that the two-level sample shape convergence and the two-level bpdd Gromov-weak convergence are equivalent on $\T_2^{(2)}$. 

Recall from \eqref{e:TLsubtri} the space $\mathfrak{D}$ of pairs $(C,K)\in\CT\times\CM_1(\CM_1(\BS))$ such that $M_K=\lambda_\BS$.

\begin{proposition}\label{p:continuitytau}
Let $\CT$ be equipped with the Hausdorff metric topology, $\CM_1(\CM_1(\BS))$ with the weak topology, $\CT\times\CM_1(\CM_1(\BS))$ with the product topology and $\T_2^{(2)}$ with the two-level sample shape topology. Then the coding map $\tau:\mathfrak{D}\to\T_2^{(2)}$ is continuous.
\end{proposition}

We will use the following lemma in the proof.

\begin{lemma}\label{l:coupledsampling}
Let $(S,d)$ be separable and let $P,Q\in\CM_1(S)$, $\epsilon>0$ such that $d_{\mathrm{Pr}}(P,Q)\leq\epsilon$ where $d_{\mathrm{Pr}}$ denotes the Prokhorov distance. Then we can define two random variables $X$, resp.\ $Y$, distributed according to $P$, resp.\ $Q$, on the same probability space $(\Omega,\PP)$ such that
\begin{equation}
	\PP\{d(X,Y)\geq\epsilon\}\leq\epsilon.
\end{equation}
\end{lemma}

\begin{proof}
The lemma directly follows from \cite[Theorem~3.1.2]{EthierKurtz1986}:
\begin{equation}
	d_{\mathrm{Pr}}(P,Q)=\inf_{\gamma}\inf\big\{\epsilon>0\big|\gamma\{(x,y)|d(x,y)\geq\epsilon\}\leq\epsilon\big\},
\end{equation}
where the first infimum is taken over the set of measures $\gamma\in\CM_1(S\times S)$ with marginals $P$ and $Q$.
\end{proof}

We can now proceed with the proof of the continuity of the coding map.

\begin{proof}[Proof of Proposition~\ref{p:continuitytau}]
Fix $\Gamma=(C,K)\in\mathfrak{D}$, $m\in\N$ and $\vn\in\N^m$. By Remark~\ref{r:treeshapedistribution}, it is enough to show that $\mathfrak{S}_{m,\vn}\circ\tau:\CT\to\CM_1(\FC_{m,\vn})$ is continuous at $(C,K)$.

Let $\kappa_1,...,\kappa_m$ be independent, identically distributed measures on $\CM_1(\BS)$ with distribution $K$. For all $1\leq i\leq m$, let $U_{i1},...,U_{in_i}$ be independent, identically distributed points on $\BS$ with distribution $\kappa_i$. Fix $\epsilon>0$. Since $M_K=\lambda_\BS$, there exist $N\in\N$ and $v_1,...,v_N\in\Delta(C)\cap\nabla(C)$ distinct, such that with probability at least $1-\epsilon$ the following holds:
\begin{itemize}
	\item if $\{U_{ij}:(i,j)\in\N_{m,\vn}\}\cap\{v\}\neq\emptyset$ for some $v\in\nabla(C)$, then $v\in\{v_1,...,v_N\}$, and
	\item if $\{U_{ij}:(i,j)\in\N_{m,\vn}\}\cap\comp_v(w)\neq\emptyset$ for some $v\in\Delta(C)$ and all $w\in\Delta(C)\cup\nabla(C)\cup\square(C)$ with $w\neq v$, then $v\in\{v_1,...,v_N\}$.
\end{itemize}
Put $\epsilon':=\epsilon(12N|\vn|)^{-1}$. Then, since $M_K=\lambda_\BS$, each $U_{ij}$ is distributed according to the Lebesgue measure on $\BS$, so that
\begin{equation}\label{e:majdisttoC}
\begin{aligned}
	\PP & \big(\big\{d(U_{ij},\partial v_k)\geq\epsilon',\forall(i,j)\in\N_{m,\vn},k=1,...,N\big\}\big)\\
	&\geq 1-\sum_{(i,j)\in\N_{m,\vn}}\PP\big(\big\{d(U_{ij},\partial v_k)\leq\epsilon',\forall k=1,...,N\big\}\big)\\
	&\geq 1-|\vn|\PP\big(\big\{d(U_{11},\partial v_k)\leq\epsilon',\forall k=1,...,N\big\}\big)\\
	&\geq 1-\epsilon.
\end{aligned}
\end{equation}

Now, there is a $\delta=\delta(\epsilon)>0$ small enough such that for any $C'\in\CT$ satisfying $d_H(C,C')<\delta$ there are distinct $v'_1,...,v'_N\in\Delta(C')\cup\nabla(C')$ such that $d_H(v_k,v'_k)\leq \frac{\epsilon'}{2}$ for $k=1,...,N$. Let $K'\in\CM_1(\CM_1(\BS))$ such that $M_K=\lambda_\BS$. Suppose that $d_{\mathrm{Pr}}(K,K')\leq \epsilon'':=\min(\frac{\epsilon}{2|\vn|},\frac{\epsilon'}{2})$. By Lemma~\ref{l:coupledsampling}, there exist $\kappa'_1,...,\kappa'_m$ independent, identically distributed measures on $\CM_1(\BS)$ with distribution $K'$ coupled to $\kappa_1,...,\kappa_m$ such that for all $i=1,...,m$, $\PP\{d_{\mathrm{Pr}}(\kappa_i,\kappa'_i)\geq\epsilon''\}\leq \epsilon''$. Applying again Lemma~\ref{l:coupledsampling} for each $1\leq i\leq m$, there exist $U'_{i1},...,U'_{in_i}$ independent, identically distributed points on $\BS$ with distribution $\kappa'_i$ coupled to $U_{i1},...,U_{in_i}$ such that for all $j=1,...,n_i$,
\begin{equation}
	\PP\{d(U_{ij},U'_{ij})\leq\epsilon''\}=\PP\big\{d(U_{ij},U'_{ij})\leq\epsilon''\big|d_{\mathrm{Pr}}(\kappa_i,\kappa'_i)\leq\epsilon''\big\}\PP\{d_{\mathrm{Pr}}(\kappa_i,\kappa'_i)\leq\epsilon''\}\geq 1-2\epsilon''.
\end{equation}
Therefore,
\begin{equation}
	\PP\big(\big\{d(U_{ij},U'_{ij})\leq\epsilon'',\forall(i,j)\in\N_{m,\vn}\big\}\big)\geq 1-2|\vn|\epsilon''.
\end{equation}
Thus, since $\epsilon''\leq\frac{\epsilon'}{2}$, and using \eqref{e:majdisttoC},
\begin{equation}
	\PP\left(\left\{d(U'_{ij},\partial v_k)\geq\frac{\epsilon'}{2},\forall(i,j)\in\N_{m,\vn},k=1,...,N\right\}\right)\geq 1-2\epsilon.
\end{equation}

Let $\chi=(T,c,\nu):=\tau(C,K)$ and $(V_{ij})_{(i,j)\in\N_{m,\vn}}$ be distributed according to
\begin{equation}
	M_\nu^{m,\vn}(\cdot):=\int\nu^{\otimes m}(\du{\mu})\int\bigotimes_{i=1}^m\mu_i^{\otimes n_i}(\cdot),
\end{equation}
coupled to $(U_{ij})_{(i,j)\in\N_{m,\vn}}$ such that $V_{ij}\in\CS_v(w)$ if and only if $U_{ij}\in\comp_v(w)$. This is possible due to the properties of the coding map $\tau$ stated in Theorem~\ref{t:codingmap}. Define $\chi'$ and $(V'_{ij})_{(i,j)\in\N_{m,\vn}}$ similarly with $(C,K)$ replaced by $(C',K')$. Then
\begin{equation}
	\PP\big(\big\{\Fs_{(T,c)}(\mat{V})=\Fs_{(T',c')}(\mat{V'})\big\}\big)\geq 1-2\epsilon.
\end{equation}
Therefore,
\begin{equation}
	d_{\mathrm{Pr}}\big(\mathfrak{S}_{m,\vn}(\tau(C,K)),\mathfrak{S}_{m,\vn}(\tau(C',K'))\big)\leq2\epsilon.
\end{equation}
We proved that $\mathfrak{S}_{m,\vn}\circ\tau$ is continuous at $(C,K)$, with $m,\vn$ and $(C,K)$ arbitrary. Therefore, the coding map $\tau$ is continuous.
\end{proof}

We finish this section with our second main result.

\begin{theorem}[Equivalence of topologies and compactness]\label{t:equivalencecompactness}
The two-level sample shape topology and the two-level bpdd-Gromov-weak topology coincide on $\T_2^{(2)}$. Furthermore, $\T_2^{(2)}$ is compact and metrizable in this topology.
\end{theorem}

\begin{proof}
We start by showing that the two-level sample shape topology on $\T_2^{(2)}$ is compact. The set $\CT$ equipped with the Hausdorff metric topology is compact (see \cite[Lemma~4.2]{LoehrWinter2021}). Moreover, the circle line $\BS$ is compact so that $\CM_1(\BS)$ is compact and $\CM_1(\CM_1(\BS))$ as well. Since the subset of two-level measures $K$ such that $M_K=\lambda_\BS$ is closed in $\CM_1(\CM_1(\BS))$, $\mathfrak{D}$ is a compact space. By Theorem~\ref{t:codingmap}, the coding map is surjective and by Proposition~\ref{p:continuitytau}, it is continuous when $\T_2^{(2)}$ is equipped with the two-level sample shape topology. Therefore, the sample shape topology is a compact topology.

Furthermore, the two-level bpdd-Gromov-weak topology is a Hausdorff topology by Proposition~\ref{p:separablemetrizable} and weaker than the two-level sample shape topology by Proposition~\ref{p:strongertopology}. Thus, both topologies coincide on $\T_2^{(2)}$, and we know from Corollary~\ref{c:metrizability} that is is metrizable.
\end{proof}

\section{Example: the Kingman algebraic two-level measure tree}
\label{S:Kingman}

One of the first stochastic models describing the genealogy of a large population is the Kingman coalescent. It is a partition-valued process introduced in 1982 \cite{Kingman1982} to model an evolutionary tree by looking at the genealogy backwards in time. There is not hierarchy in the original Kingman coalescent. Independently and with the same rate, each pair of individuals coalesce, meaning that their partitions merge. A natural extension of this model takes into account that the population consists of several species. Both at the individual and species
level, the coalescences follow the rule of the classical Kingman coalescent. Pairs of individuals belonging to the same species can merge as before, but individuals belonging to different species cannot coalesce before their species themselves have coalesced. Such process, named as the nested Kingman coalescent, has been considered by \cite{BlancasBenitezDuchampsLambertSiriJegousse2018,BlancasBenitezRogersSchweinsbergSiriJegousse2019,Dawson2018} or \cite{Meizis2019} for instance.

We introduce here the Kingman algebraic two-level measure tree, which corresponds to the nested Kingman coalescent measure tree, as defined in \cite{Meizis2019}, without branch length. For this, we rely on a sampling consistency of the nested Kingman coalescent and the compactness of $\T_2^{(2)}$. We first recall a definition of the nested Kingman coalescent for a host-parasite population indexed in $\N^2$, that is, $(i,j)$ is the $j$-th parasite in the $i$-th host.

\josue{For $I\subseteq\N$, let $\CE(I)$ be the set of equivalence relations on $I$. The equivalence classes of an equivalence relation are called \emph{blocks}. Let
\begin{equation}
	\CP_{nest}(I)\subset\CE(I)\times\CE(I)
\end{equation}
denote the set of all pairs $(\Fp_H,\Fp_P)$ of \emph{nested} equivalence relations on $I$, i.e.\ such that
\begin{itemize}
	\item if $i_1=i_2$, then $(i_1,j_1)$ and $(i_2,j_2)$ belong to the same block of $\Fp_H$, that is, all parasites of the same host belong to the same block of $\Fp_H$,
	\item if $(i_1,j_1)$ and $(i_2,j_2)$ belong to the same block of $\Fp_P$, then they belong to the same block of $\Fp_H$. In other words, each block of $\Fp_P$ is contained in a single block of $\Fp_H$. This property is the reason to call it ``nested''.
\end{itemize}
Therefore, $\Fp_H$ represents the population of hosts and $\Fp_P$ the population of parasites.

We also define the following equivalence relations on $\N^2$:
\begin{equation}
	\begin{aligned}
	P_0 & := \big\{\big((i,j),(i,j)\big)\big|(i,j)\in\N^2\big\}\\
	H_0 & :=\big\{\big((i,j),(i,k)\big)\big|i,j,k\in\N\big\},
	\end{aligned}
\end{equation}
which will be the initial states of the nested Kingman coalescent.}

\begin{definition}[Finite nested Kingman coalescent]
Let $I$ be a finite subset of $\N^2$ and $\gamma_H,\gamma_P>0$. The \emph{finite nested Kingman coalescent
\begin{equation}
	(\KK^I(t))_{t\geq 0}=\big(\kappa^I_H(t),\kappa^I_P(t)\big)_{t\geq 0}
\end{equation}
on $I$ with rates $(\gamma_H,\gamma_P)$} is a continuous-time Markov process with values in $\CP_{nest}(I)$ such that:
\begin{enumerate}
	\item The initial state is $\KK^I(0)=(H_0\cap I^2,P_0\cap I^2)$.
	\item $(\kappa^I_H(t))_{t\geq 0}$ is a Kingman coalescent with rate $\gamma_H$, i.e.\ any pair of blocks in $\kappa^I_H(t)$ merge at rate $\gamma_H$.
	\item $(\kappa^I_P(t))_{t\geq 0}$ behaves in the following way: any pair of blocks $\pi_1,\pi_2$ of $\kappa^I_P(t)$ such that $\pi_1\cup\pi_2$ is contained in a single block of $\kappa^I_H(t)$ merge at rate $\gamma_P$. Other blocks of $\kappa^I_P(t)$ cannot merge.
\end{enumerate}
\end{definition}

Roughly speaking, hosts merge as in a Kingman coalescent with rate $\gamma_H$, and parasites within the same host merge as in a Kingman coalescent with rate $\gamma_P$. Since this process has only finitely many states, it is well defined and unique. Furthermore, it satisfies an important property:

\begin{proposition}[Sampling consistency]\label{p:samplingconsis}
Let $I$ and $J$ be two finite subsets of $\N^2$ such that $J\subset I$ and let $\gamma_H,\gamma_P>0$. The restriction of $\KK^I$ to $J$ has same distribution as $\KK^J$, that is, the restriction of $\KK^I$ to $J$ is a finite nested Kingman coalescent on $J$ with rates $(\gamma_H,\gamma_P)$.
\end{proposition}

Using this property, we can show the existence of the nested Kingman coalescent for an infinite set of hosts and of parasites (see \cite[Section~5]{BlancasBenitezDuchampsLambertSiriJegousse2018} for a construction of more general nested coalescents).

\begin{definition}[Nested Kingman coalescent]
Let $\gamma_H,\gamma_P>0$. The \emph{nested Kingman coalescent $\KK$ with rates $(\gamma_H,\gamma_P)$} is a continuous-time Markov process with values in $\CP_{nest}(\N^2)$ such that for any finite subset $I$ of $\N^2$, the restriction of $\KK$ to $I$ is a finite nested Kingman coalescent on $I$ with rates $(\gamma_H,\gamma_P)$.
\end{definition}

\josue{In order to define the Kingman algebraic two-level measure tree, we first introduce the nested Kingman \emph{rooted} algebraic tree. For this, let us shortly explain how algebraic trees are extended to rooted algebraic trees.}

\begin{definition}[Rooted and partially ordered algebraic tree]
A \emph{partially ordered algebraic tree} is a non-empty set $T$ together with
%a distinguished point $\rho\in T$ and
a symmetric map $c_{\wedge}:\,T\times T\to T$ such that:
\begin{enumerate}
\item[(M1)] For all $x\in T$,
  $c_{\wedge}\big(x,x\big)=x.$
\item[(M2)] For all $x_1,x_2,x_3\in T$,
  $c_{\wedge}\big(x_1,c_{\wedge}(x_2,x_3)\big)=c_{\wedge}(c_{\wedge}(x_1,x_2),x_3).$
\item[(M3)] For all $x_1,x_2,x_3\in T$, $\#\{c_{\wedge}(x_1,x_2),c_{\wedge}(x_1,x_3),c_{\wedge}(x_2,x_3)\}\le 2$ and if $c_{\wedge}(x_1,x_2)=c_{\wedge}(x_1,x_3)$, then
\begin{equation}
\label{e:M3}
   c_{\wedge}(x_1,x_2)=c_{\wedge}\big(c_{\wedge}(x_1,x_2),c_{\wedge}(x_2,x_3)\big).
\end{equation}
\end{enumerate}
We refer to $c_{\wedge}$ as the {\em minimum map}.

A {\em rooted algebraic tree} $(T,c_\wedge)$ is a partially ordered algebraic tree for which there exists a point $\rho\in T$ with
$c_\wedge(\rho,x)=\rho$ for all $x\in T$. We will refer to (this unique) $\rho$ as the {\em root} of $(T,c_\wedge)$.
\label{d:rootedtrees}
\end{definition}

\begin{remark}[(M1) and (M2) define a partial order] Let $(T,c_{\wedge})$ be a partially ordered algebraic tree. In what follows, we write for $x,y\in T$, $x\le y$ if and only if $x=c_{\wedge}(x,y)$ respectively, $x< y$ if and only if $x=c_{\wedge}(x,y)\not =y$. Notice that the first two conditions (M1) and (M2) ensure that $\le$ defines a partial order relation. Indeed, reflexivity follows from (M1), antisymmetry follows from the fact that $c_{\wedge}$ is a symmetric map and transitivity follows from (M2), i.e., if $x,y,z\in T$ are such that $x\le y$ and $y\le z$, or equivalently, if $x=c_{\wedge}(x,y)$ and $y=c_{\wedge}(y,z)$, then $x=c_{\wedge}(x,y)=c_{\wedge}(x,c_{\wedge}(y,z))=c_{\wedge}(c_{\wedge}(x,y),z)=c_{\wedge}(x,z)$.
\label{Rem:001}
\end{remark}

When we add the third condition (M3) to the partially ordered set defined by (M1) and (M2), we ensure that there are no loops, so that a rooted algebraic tree defines an (unrooted) algebraic tree. Reciprocally, by distinguishing a point in an algebraic tree, we can define a rooted algebraic tree (see \cite[Section~2.2]{NussbaumerWinter2020} for more details).

\begin{proposition}[Rooted versus unrooted algebraic trees] Let $T\not=\emptyset$.
\begin{itemize}
	\item[(i)] If $(T,c_\wedge)$ is a partially ordered, algebraic tree and $c:T^3\to T$ the symmetric map defined as
	\begin{equation}
	\begin{aligned}
	   c(x,y,z)
	   &:=
	   \max\{c_\wedge(x,y),c_\wedge(x,z),c_\wedge(y,z)\},\hspace{.5cm}x,y,z\in T,
	\end{aligned}
	\end{equation}
	then $(T,c)$ is an algebraic tree.
	\item[(ii)] If $(T,c)$ is an algebraic tree, $\rho\in T$, and $c_\wedge:T^2\to T$ the symmetric map defined as
	\begin{equation}
	   c_\wedge(x,y):=c(x,y,\rho),\hspace{.5cm}x,y\in T,
	\end{equation}
	then $(T,c_\wedge,\rho)$ is a rooted algebraic tree.
\end{itemize}
\label{p:rootedunrooted}
\end{proposition}

We now construct the Kingman a2m tree with rates $(\gamma_H,\gamma_P)$. Given $\KK=(\kappa_H,\kappa_P)$, we define the {\em nested Kingman rooted algebraic tree} as the random rooted algebraic tree $(T,c_\wedge,\rho)$ with the vertex set
\begin{equation}
   T:=\{\varpi_\rho\}\uplus\bigcup_{t\ge 0}\bigcup_{\varpi\in \kappa_P(t)}\varpi,
\end{equation}
with $\varpi_\rho=\N^2\uplus\{\rho\}$ for a point $\rho\not\in\N^2$ and the minimal map $c_\rho$ which sends two elements $\varpi,\varpi'\in T$ to the smallest $\tilde{\varpi}\in T$ which contains both $\varpi$ and $\varpi'$, i.e.,
\begin{equation}
   c_\wedge(\varpi,\varpi')
 :=
   \bigcap_{\tilde{\varpi}\in T:\,\varpi,\varpi'\subseteq\tilde{\varpi}}\widetilde{\varpi}.
\end{equation}
Further, we define the {\em nested Kingman algebraic tree} as the random algebraic space $(T,c)$ obtained from the rooted nested Kingman algebraic tree $(T,c_\wedge,\varpi_\rho)$ as in (\ref{p:rootedunrooted}).

For $M\in\N$ and $\vN\in\N^M$, define also the two-level measure
\begin{equation}
   \nu^{M,\vN}:=\frac{1}{M}\sum_{i=1}^M\delta_{\frac{1}{N_i}\sum_{j=1}^{N_i}\delta_{\{(i,j)\}}},
\end{equation}
so that $(T,c,\nu^{M,\vN})$ is a random binary a2m tree in $\T^{(2)}_2$. We define the Kingman a2m tree as the weak limit (with respect to the two-level sample shape convergence) of these random trees as $M,\vN\to\infty$, that is, $M$ and $\inf_{i\in M}N_i$ simultaneously go to infinity.

\begin{proposition}[Kingman algebraic two-level measure tree] For $M\in\N$ and $\vN\in\N^M$, let $\chi^{M,\vN}=(T,c,\nu^{M,\vN})$. Then their exists a random binary a2m tree $\chi\in\T_2^{(2)}$ such that
\begin{equation}
   \chi^{M,\vN}{_{\displaystyle\Longrightarrow\atop M,\vN\to\infty}} \chi,
\end{equation}
where $\Rightarrow$ stands for weak convergence on $\mathbb{T}_2^{(2)}$ equipped with the two-level sample shape convergence.
\end{proposition}

\begin{proof} By Proposition~\ref{t:equivalencecompactness}, the space $\mathbb{T}_2^{(2)}$ equipped with the two-level sample shape convergence is compact. Therefore the sequence $\{\chi^{M,\vN}:\,M\in\N,\vN\in\N^M\}$ is clearly tight.

The {\em uniqueness} of the limit follows from the sampling consistency of the family of finite nested Kingman coalescents. Fix $M\in\N$ and $\vN\in\N^M$, and $m\in\N$, $\vn\in\N^M$ such that $m\leq M$ and for each $1\leq i\leq m$, $n_i\leq N_i$. Consider $\{U_{ij}:(i,j)\in\N_{m,\vn}\}$ sampled (without repetition) at random from $\chi^{M,\vN}$. By Proposition~\ref{p:samplingconsis}, the restriction of $\KK^{\N_{M,\vN}}$ to $\{U_{ij}:(i,j)\in\N_{m,\vn}\}$ is a finite nested Kingman coalescent on $\{U_{ij}:(i,j)\in\N_{m,\vn}\}$. Therefore, the shape $\Fs_{(T,c)}(\mat{U})\in\FC_{m,\vn}$ of the subtree spanned by the leaves $\{U_{ij}:(i,j)\in\N_{m,\vn}\}$ has the distribution of the shape of $\chi^{m,\vn}$. Since sampling with and without repetition is asymptotically equivalent when $M,\vN\to\infty$, the claim follows by definition of the two-level sample shape convergence.
\end{proof}

We can now conclude the paper with the construction of the Kingman algebraic two-level measure tree, which is obtained as the limit of finite nested Kingman coalescent trees and which is the algebraic tree corresponding to the two-level Kingman coalescent introduced in Meizis \cite{Meizis2019}.

\begin{definition}[Kingman algebraic two-level measure tree]
The \emph{Kingman algebraic two-level measure tree $\chi$ with rates $(\gamma_H,\gamma_P)$} is the unique limit in $\T^{(2)}_2$ of the sequence $(\chi^{M,\vN})_{M,\vN}$, where $\chi^{M,\vN}$ is the random algebraic two-level measure tree obtained from the finite nested Kingman coalescent on $\N_{M,\vN}$ with rates $(\gamma_H,\gamma_P)$.
\end{definition}

\appendix

\josue{\textsc{Acknowledgements.} This work was done as part of the PhD thesis of Josué Nussbaumer, which was a joint program between Université de Lille and Universität Duisburg-Essen and supported by Université franco-allemande (UFA). J.N. and A.W. would like to thank the DFG through the SPP Priority Programme 1590. J.N. and V.C.T. were supported by Labex B\'ezout (ANR-10-LABX-58) and ``Chaire Mod\'elisation Math\'ematique et Biodiversit\'e (MMB)'' of Veolia Environnement-Ecole Polytechnique-Museum National d'Histoire Naturelle-Fondation X.}

%\bibliography{bibliography}
%\bibliographystyle{alpha}
\end{document}

%% file: figures.tex
\psset{linewidth=0.05pt, dotsize=3.6pt, hatchsep=0.8pt, hatchwidth=0.6pt}
\newcommand{\deffillstyle}{vlines}
\newcommand{\stylecomp}{hlines}
\newcommand{\dotcircle}{\pscircle[linestyle=dotted,linewidth=.75pt](0,0){1}}
\newcommand{\dotarc}[2]{\psarc[linestyle=dotted,linewidth=.75pt]{-}(0,0){1}{#1}{#2}}
\newcommand{\filledarc}[2]{
	\psline(1;#1)(1;#2)
	\psarc[fillstyle=\deffillstyle]{-}(0,0){1}{#1}{#2}
}

\newcommand{\figureconditionMK}{
\psset{unit=0.139\textwidth}
\begin{pspicture}(-1.01,-1.2)(1.01,1.01)
	\SpecialCoor
	\pspolygon(1;-90)(1;30)(1;-210)
	\dotarc{0}{360}
	\psdot[dotsize=6pt](1;-90)
	\rput(1.15;-90){0}
\end{pspicture}
\hfil
\begin{pspicture}(-1.01,-1.2)(1.01,1.01)
	\SpecialCoor
	\pspolygon(1;-90)(1;30)(1;-210)
	\dotarc{0}{360}
	\psdot[dotsize=3pt](1;-90)
	\multido{\n=2+1}{100}{ \psdot[dotsize=3pt](1;-\fpeval{trunc(90+360/\n,3)}) }
	\rput(1.15;-90){0}
	\rput(1.15;90){$-\frac{1}{2}$}
	\rput(1.2;150){$-\frac{1}{3}$}
	\rput(1.2;180){$-\frac{1}{4}$}
	\rput(1.2;198){$-\frac{1}{5}$}
	\rput(1.2;210){$-\frac{1}{6}$}
\end{pspicture}}

\newcommand{\figurefillingofedges}{
\psset{unit=0.09\textwidth}
\begin{pspicture}(-2.5,-1)(2.5,1.5)
	\SpecialCoor
	\psline[linestyle=dotted,linewidth=1.25pt](-1,0)(1,0)
	\psline[linestyle=dotted,linewidth=1.25pt](-1,0)(-2,1)
	\psline(-1,0)(-2,-1)
	\psline(1,0)(2,1)
	\psline[linewidth=2pt](1,0)(2,-1)
	\psdot(-2,1)
	\psdot(-1,0)
	\rput(-2,1.15){$v$}
	\rput(-0.95,0.15){$w$}
\end{pspicture}
\hfil
\begin{pspicture}(-2.5,-1)(2.5,1.5)
	\SpecialCoor
	\psline[linestyle=dotted,linewidth=1.25pt](-1,0)(1,0)
	\psline(-1,0)(-2,1)
	\psline(-1,0)(-2,-1)
	\psline(1,0)(2,1)
	\psline[linewidth=2pt](1,0)(2,-1)
	\psdot(-2,1)
	\psdot(-1,0)
	\rput(-2,1.15){$v$}
	\rput(-0.95,0.15){$w$}
	\rput[l](-1.45,0.6){$\{(v,w)\}\times(0,1)$}
\end{pspicture}}

\newcommand{\figureexamplesubtriangulation}{
\psset{unit=0.145\textwidth}
\begin{pspicture}(-1.01,-1.01)(1.01,1.01)
	\SpecialCoor
	\pspolygon(1;45)(1;-90)(1;0)
	\pspolygon(1;135)(1;-90)(1;45)
	\pspolygon(1;135)(1;-90)(1;225)
	\dotarc{0}{360}
	\psline(1;90)(1;135)
	\filledarc{135}{225}
	\pspolygon[fillstyle=\deffillstyle,linestyle=none](1;45)(1;135)(1;90)
	\psarc[fillstyle=\deffillstyle]{-}(0,0){1}{45}{90}
	\psdot(1;180)
	\rput(1.15;180){$p$}
\end{pspicture}}

\newcommand{\figuredefinitioncomp}{
\psset{unit=0.1\textwidth}
\begin{pspicture}(-1.5,-1.2)(1.5,1.2)
	\SpecialCoor
	\pspolygon(1;45)(1;-90)(1;0)
	\pspolygon(1;135)(1;-90)(1;45)
	\pspolygon(1;135)(1;-90)(1;225)
	\dotarc{0}{360}
	\psline(1;90)(1;135)
	\filledarc{135}{225}
	\pspolygon[fillstyle=\deffillstyle,linestyle=none](1;45)(1;135)(1;90)
	\psarc[fillstyle=\deffillstyle]{-}(0,0){1}{45}{90}
	\psline[linewidth=1pt](1;157.5)(1;202.5)
	\rput(0.1;90){$x$}
	\rput(1.15;202.5){$y$}
\end{pspicture}
\hfil
\begin{pspicture}(-1.5,-1.2)(1.5,1.2)
	\SpecialCoor
	\pspolygon(1;45)(1;-90)(1;0)
	\pspolygon(1;135)(1;-90)(1;45)
	\pspolygon(1;135)(1;-90)(1;225)
	\dotarc{0}{360}
	\psline(1;90)(1;135)
	\filledarc{135}{225}
	\pspolygon[fillstyle=\deffillstyle,linestyle=none](1;45)(1;135)(1;90)
	\psarc[fillstyle=\deffillstyle]{-}(0,0){1}{45}{90}
	\psline[linewidth=1pt](1;135)(1;67.5)
	\rput(1.15;67.5){$x$}
	\rput(0.85;-45){$y$}
\end{pspicture}
\hfil
\begin{pspicture}(-1.5,-1.2)(1.5,1.2)
	\SpecialCoor
	\pspolygon(1;45)(1;-90)(1;0)
	\pspolygon(1;135)(1;-90)(1;45)
	\pspolygon(1;135)(1;-90)(1;225)
	\dotarc{0}{360}
	\psline(1;90)(1;135)
	\filledarc{135}{225}
	\pspolygon[fillstyle=\deffillstyle,linestyle=none](1;45)(1;135)(1;90)
	\psarc[fillstyle=\deffillstyle]{-}(0,0){1}{45}{90}
	\psline[linewidth=1pt](1;135)(1;67.5)
	\rput[l](1.1;67.5){$x=y$}
\end{pspicture}
\hfil
\begin{pspicture}(-1.5,-1.2)(1.5,1.2)
	\SpecialCoor
	\psarc[fillstyle=\stylecomp]{-}(0,0){1}{135}{270}
	\dotarc{270}{135}
\end{pspicture}
\hfil
\begin{pspicture}(-1.5,-1.2)(1.5,1.2)
	\SpecialCoor
	\psarc[fillstyle=\stylecomp]{-}(0,0){1}{135}{67.5}
	\dotarc{67.5}{135}
\end{pspicture}
\hfil
\begin{pspicture}(-1.5,-1.2)(1.5,1.2)
	\SpecialCoor
	\psline[linewidth=1pt](1;135)(1;67.5)
	\dotarc{0}{360}
\end{pspicture}}

\newcommand{\figuretotalorder}{
\psset{unit=0.12\textwidth, arrowsize=4pt}
\begin{pspicture}(-0.75,-0.75)(1.25,1.75)
	\SpecialCoor
	\psline(0,0)(0,-.5)
	\psline(0,0)(-.5,.5)
	\psline[linewidth=1.5pt](0,0)(.5,.5)
	\psline(.5,.5)(0,1)
	\psline[linewidth=1.5pt](.5,.5)(1,1)
	\psline(0,1)(-.5,1.5)
	\psline[linewidth=1.5pt](0,1)(0.5,1.5)
	\psdot[dotsize=4pt](0,-.5)
	\psdot[dotsize=4pt](-.5,.5)
	\psdot[dotsize=4pt](1,1)
	\psdot[dotsize=4pt](-.5,1.5)
	\rput(.1,-.6){$\rho$}
	\rput(-0.1,-0.05){$v_1$}
	\rput(0.6,0.4){$v_2$}
	\rput(-0.1,.95){$v_3$}
	\rput(1.1,1.1){$w_1$}
	\rput(-0.7,1.5){$w_2$}
	\rput(-0.7,.5){$w_3$}
	\psarc{->}(-0.35,-.25){.5}{-22.5}{22.5}
	\psarc{->}(0,.5){.5}{-67.5}{-22.5}
	\psarc{->}(0.5,1){.5}{-67.5}{-22.5}
	\psarc{->}(0.5,.6){.5}{67.5}{112.5}
	\psarc{->}(-0.5,1){.5}{22.5}{67.5}
	\psarc{->}(-0.1,1){.5}{157.5}{202.5}
\end{pspicture}
\hfil
\psset{unit=0.145\textwidth}
\begin{pspicture}(-1.01,-1.01)(1.01,1.01)
	\SpecialCoor
	\pspolygon(1;-90)(1;-135)(1;0)
	\pspolygon(1;135)(1;-135)(1;45)
	\pspolygon(1;135)(1;-135)(1;180)
	\dotarc{0}{360}
	\psline(1;90)(1;135)
	\filledarc{135}{180}
	\pspolygon[fillstyle=\deffillstyle,linestyle=none](1;45)(1;0)(1;-135)
	\psarc[fillstyle=\deffillstyle]{-}(0,0){1}{0}{45}
	\pspolygon[fillstyle=\deffillstyle,linestyle=none](1;45)(1;135)(1;90)
	\psarc[fillstyle=\deffillstyle]{-}(0,0){1}{45}{90}
	\rput(1.1;-45){$\rho$}
	\rput(0.7;-90){$v_1$}
	\rput(0.4;135){$v_2$}
	\rput(0.85;180){$v_3$}
	\rput(1.1;115){$w_1$}
	\rput(1.1;202.5){$w_2$}
	\rput(1.1;-115){$w_3$}
	\psline(0.97;-90)(1.03;-90)
	\rput(1.1;-90){0}
	\psline(0.97;0)(1.03;0)
	\rput(1.3;0){$M_\nu(\rho)$}
\end{pspicture}}

\newcommand{\figuretypesofboundaries}{
\psset{unit=0.18\textwidth}
\begin{pspicture}(-1.01,-1.01)(1.51,1.01)
	\SpecialCoor
	\psline[linewidth=2pt](0;0)(1;150)
	\psline(0.6;146)(.6;154)
	\psdot[dotsize=5pt](1;150)
	\psline[linewidth=2pt](0;0)(1;270)
	\psline(0.3;262)(.3;278)
	\psline(0;0)(.5;30)
	\multido{\n=-18+36}{2}{ \psline(.5;30)(.75;\fpeval{30+\n})
	\multido{\ne=-9+18}{2}{ \psline(.75;\fpeval{30+\n})(.875;\fpeval{30+\n+\ne})
	\multido{\nee=-4.5+9}{2}{ \psline(.875;\fpeval{30+\n+\ne})(.9375;\fpeval{30+\n+\ne+\nee})
	\psline[linestyle=dotted](.9375;\fpeval{30+\n+\ne+\nee})(1.125;\fpeval{30+\n+\ne+\nee})
	}
	}
	}
	\rput(1.15;150){$x_{i1}$}
	\rput(.35;300){$x_{i2}$}
	\rput(.6;135){$x_{i3}$}
	\psdot[dotsize=3pt](1.135;43)
	\rput(1.25;42){$x_{i5}$}
	\rput(1.1;-90){$x_{i4}$}
\end{pspicture}
\hfil
\begin{pspicture}(-1.51,-1.01)(1.01,1.01)
	\SpecialCoor
	\pspolygon(1;30)(1;120)(1;-60)
	\dotarc{165}{210}
	\multido{\n=-60+45,\ne=-15+45}{4}{ \psline{-}(1;\n)(1;\ne) }
	\multido{\n=-60+22.5,\ne=-37.5+22.5}{8}{ \psline{-}(1;\n)(1;\ne) }
	\multido{\n=-60+11.25,\ne=-48.75+11.25}{16}{ \psline{-}(1;\n)(1;\ne) }
	\psarc{-}(0,0){1}{-60}{120}
	\pspolygon(1;210)(1;120)(1;-60)
	\filledarc{210}{300}
	\pspolygon[fillstyle=\deffillstyle,linestyle=none](1;120)(1;165)(1;210)
	\psarc[fillstyle=\deffillstyle]{-}(0,0){1}{120}{165}
	\psline(1;210)(1;165)
	\psline[linewidth=2pt](1;210)(1;147)
	\psdot[dotsize=6pt](1;147)
	\psline[linewidth=2pt](1;223.5)(1;286.5)
	\psdot[dotsize=4pt](1;223.5)
	\psdot[dotsize=4pt](1;286.5)
	\psdot[dotsize=6pt](1;65)
	\psdot[dotsize=6pt](1;255)
	\rput(1.3;192.5){$\overline{x_{i1}}\cap\mathbb{S}$}
	\rput(1.15;147){$\widetilde{x_{i3}}$}
	\rput(1.15;223.5){$\widetilde{x_{i21}}$}
	\rput(1.15;286.5){$\widetilde{x_{i22}}$}
	\rput(1.15;255){$\widetilde{x_{i4}}$}
	\rput(1.15;65){$\widetilde{x_{i5}}$}
\end{pspicture}}

\newcommand{\figurefilledareas}{
\psset{unit=0.139\textwidth}
\begin{pspicture}(-1.01,-1.01)(1.01,1.01)
	\SpecialCoor
	\pspolygon[linewidth=1pt](1;45)(1;90)(1;180)
	\psline[linewidth=1pt]{-}(1;225)(1;270)
	\dotarc{0}{360}
	\multido{\n=182+2,\ne=39+-6}{21}{ \psline[linewidth=1pt]{-}(1;\n)(1;\ne) }
	\psarc[linewidth=1pt]{-}(0,0){1}{-90}{45}
	\psarc[linewidth=1pt]{-}(0,0){1}{180}{225}
	\pspolygon[fillstyle=\deffillstyle,linestyle=none](1;45)(1;180)(1;225)(1;270)
	\psarc[fillstyle=\deffillstyle,linestyle=none]{-}(0,0){1}{180}{225}
	\psarc[fillstyle=\deffillstyle,linestyle=none]{-}(0,0){1}{270}{45}
	\rput(1.15;45){$x$}
	\rput(1.15;270){$y$}
	\rput(1.15;225){$y'$}
	\rput(1.15;180){$x'$}
\end{pspicture}
\hfil
\begin{pspicture}(-1.01,-1.01)(1.01,1.01)
	\SpecialCoor
	\pspolygon[linewidth=1pt](1;45)(1;270)(1;180)
	\dotarc{0}{360}
	\multido{\n=274+4,\ne=41+-4}{17}{ \psline[linewidth=1pt]{-}(1;\n)(1;\ne) }
	\psarc[linewidth=1pt]{-}(0,0){1}{-90}{45}
	\psarc[fillstyle=\deffillstyle,linestyle=none]{-}(0,0){1}{270}{45}
	\rput(1.15;45){$x$}
	\rput[l](1.05;-22.5){$y=y'$}
	\rput(1.15;270){$x'$}
\end{pspicture}}

\newcommand{\figureunionofcomponents}{
\psset{unit=0.1\textwidth}
\begin{pspicture}(-1.05,-1.02)(1.05,1.02)
	\SpecialCoor
	\dotcircle
	\psarc[fillstyle=\stylecomp]{-}(0,0){1}{-90}{45}
	\psarc[fillstyle=\stylecomp]{-}(0,0){1}{100}{-110}
\end{pspicture}
\hfil
\begin{pspicture}(-1.05,-1.02)(1.05,1.02)
	\SpecialCoor
	\dotcircle
	\psarc[fillstyle=\stylecomp]{-}(0,0){1}{45}{-90}
	\psarc[fillstyle=\stylecomp]{-}(0,0){1}{-110}{100}
\end{pspicture}
\hfil
\begin{pspicture}(-1.05,-1.02)(1.05,1.02)
	\SpecialCoor
	\dotcircle
	\psarc[fillstyle=\stylecomp]{-}(0,0){1}{45}{-90}
	\psarc[fillstyle=\stylecomp]{-}(0,0){1}{100}{-110}
\end{pspicture}}

\newcommand{\figuretreefromngon}{
\psset{unit=0.1667\textwidth, linewidth=0.05pt, dotsize=3.6pt}
\providecommand{\pstrrootedge}{\psline[linestyle=dashed, linewidth=0.3pt, arrows=*-o]}
\providecommand{\pstrintedge}{\psline[linestyle=dashed, linewidth=0.3pt, arrows=*-]}
\providecommand{\pstrextedge}{\psline[linestyle=dashed, linewidth=0.3pt, arrows=o-]}
\providecommand{\addTriangCommand}{}
\begin{pspicture}(-1.01,-1.01)(1.01,1.01)
\addTriangCommand
\psarc[linecolor=red,linewidth=1.5pt]{-}(0,0){1}{0}{30}
\psarc[linecolor=green,linewidth=1.5pt]{-}(0,0){1}{30}{60}
\psarc[linecolor=green,linewidth=1.5pt]{-}(0,0){1}{60}{90}
\psarc[linecolor=red,linewidth=1.5pt]{-}(0,0){1}{90}{120}
\psarc[linecolor=red,linewidth=1.5pt]{-}(0,0){1}{120}{150}
\psarc[linecolor=red,linewidth=1.5pt]{-}(0,0){1}{150}{180}
\psarc[linecolor=green,linewidth=1.5pt]{-}(0,0){1}{180}{210}
\psarc[linecolor=red,linewidth=1.5pt]{-}(0,0){1}{210}{240}
\psarc[linecolor=green,linewidth=1.5pt]{-}(0,0){1}{240}{270}
\psarc[linecolor=red,linewidth=1.5pt]{-}(0,0){1}{270}{300}
\psarc[linecolor=green,linewidth=1.5pt]{-}(0,0){1}{300}{330}
\psarc[linecolor=green,linewidth=1.5pt]{-}(0,0){1}{330}{360}
\pstrintedge(-0.122008,-0.455342)(0.044658,0.166667)
\pstrintedge(0.455342,-0.455342)(-0.122008,-0.455342)
\pstrextedge(-0.965926,-0.258819)(-0.910684,0.000000)
\pstrintedge(-0.577350,0.333333)(0.044658,0.166667)
\pstrintedge(0.455342,0.788675)(0.622008,0.500000)
\pstrextedge(0.258819,0.965926)(0.455342,0.788675)
\pstrrootedge(0.622008,0.500000)(0.965926,0.258819)
\pstrextedge(0.707107,-0.707107)(0.288675,-0.744017)
\pstrintedge(0.288675,-0.744017)(0.455342,-0.455342)
\pstrintedge(-0.000000,-0.910684)(0.288675,-0.744017)
\pstrextedge(0.258819,-0.965926)(-0.000000,-0.910684)
\pstrextedge(0.707107,0.707107)(0.455342,0.788675)
\pstrextedge(-0.707107,0.707107)(-0.455342,0.788675)
\pstrextedge(-0.258819,-0.965926)(-0.000000,-0.910684)
\pstrintedge(-0.910684,0.000000)(-0.577350,0.333333)
\pstrintedge(0.044658,0.166667)(0.622008,0.500000)
\pstrextedge(-0.965926,0.258819)(-0.910684,0.000000)
\pstrintedge(-0.455342,0.788675)(-0.577350,0.333333)
\pstrextedge(-0.258819,0.965926)(-0.455342,0.788675)
\pstrextedge(-0.707107,-0.707107)(-0.122008,-0.455342)
\pstrextedge(0.965926,-0.258819)(0.455342,-0.455342)
\SpecialCoor
\pspolygon(1;210.000000)(1;240.000000)(1;360.000000)
\pspolygon(1;240.000000)(1;330.000000)(1;360.000000)
\pspolygon(1;90.000000)(1;150.000000)(1;210.000000)
\pspolygon(1;30.000000)(1;60.000000)(1;90.000000)
\pspolygon(1;30.000000)(1;90.000000)(1;360.000000)
\pspolygon(1;240.000000)(1;300.000000)(1;330.000000)
\pspolygon(1;240.000000)(1;270.000000)(1;300.000000)
\pspolygon(1;150.000000)(1;180.000000)(1;210.000000)
\pspolygon(1;90.000000)(1;210.000000)(1;360.000000)
\pspolygon(1;90.000000)(1;120.000000)(1;150.000000)
\end{pspicture}
\hfil
\providecommand{\pstreerootedge}{\psline[linewidth=0.3pt]}
\providecommand{\pstreeintedge}{\psline[linewidth=0.3pt]}
\providecommand{\pstreeextedge}{\psline[linewidth=0.3pt]}
\providecommand{\addTriangCommand}{}
\begin{pspicture}(-1.01,-1.01)(1.01,1.01)
\pstreeintedge(-0.122008,-0.455342)(0.044658,0.166667)
\pstreeintedge(0.455342,-0.455342)(-0.122008,-0.455342)
\pstreeextedge(-0.965926,-0.258819)(-0.910684,0.000000)
\pstreeintedge(-0.577350,0.333333)(0.044658,0.166667)
\pstreeintedge(0.455342,0.788675)(0.622008,0.500000)
\pstreeextedge(0.258819,0.965926)(0.455342,0.788675)
\pstreerootedge(0.622008,0.500000)(0.965926,0.258819)
\pstreeextedge(0.707107,-0.707107)(0.288675,-0.744017)
\pstreeintedge(0.288675,-0.744017)(0.455342,-0.455342)
\pstreeintedge(-0.000000,-0.910684)(0.288675,-0.744017)
\pstreeextedge(0.258819,-0.965926)(-0.000000,-0.910684)
\pstreeextedge(0.707107,0.707107)(0.455342,0.788675)
\pstreeextedge(-0.707107,0.707107)(-0.455342,0.788675)
\pstreeextedge(-0.258819,-0.965926)(-0.000000,-0.910684)
\pstreeintedge(-0.910684,0.000000)(-0.577350,0.333333)
\pstreeintedge(0.044658,0.166667)(0.622008,0.500000)
\pstreeextedge(-0.965926,0.258819)(-0.910684,0.000000)
\pstreeintedge(-0.455342,0.788675)(-0.577350,0.333333)
\pstreeextedge(-0.258819,0.965926)(-0.455342,0.788675)
\pstreeextedge(-0.707107,-0.707107)(-0.122008,-0.455342)
\pstreeextedge(0.965926,-0.258819)(0.455342,-0.455342)
\psdot[linecolor=red](1;15)
\psdot[linecolor=green](1;45)
\psdot[linecolor=green](1;75)
\psdot[linecolor=red](1;105)
\psdot[linecolor=red](1;135)
\psdot[linecolor=red](1;165)
\psdot[linecolor=green](1;195)
\psdot[linecolor=red](1;225)
\psdot[linecolor=green](1;255)
\psdot[linecolor=red](1;285)
\psdot[linecolor=green](1;315)
\psdot[linecolor=green](1;345)
\SpecialCoor
\end{pspicture}}

\newcommand{\figureIexamplesubtriangulation}{
\psset{unit=0.145\textwidth}
\begin{pspicture}(-1.01,-1.01)(1.01,1.01)
	\SpecialCoor
	\pspolygon(1;45)(1;-90)(1;0)
	\pspolygon(1;135)(1;-90)(1;45)
	\pspolygon(1;135)(1;-90)(1;225)
	\dotarc{0}{360}
	\psline(1;90)(1;135)
	\filledarc{135}{225}
	\pspolygon[fillstyle=\deffillstyle,linestyle=none](1;45)(1;135)(1;90)
	\psarc[fillstyle=\deffillstyle]{-}(0,0){1}{45}{90}
	\psdot(1;180)
	\rput(1.15;180){$p$}
	\rput(.1;90){$c$}
	\rput(1.1;247.5){$x$}
	\rput(1.1;112.5){$y$}
	\rput(.7;5){$z$}
\end{pspicture}
\hfil
\begin{pspicture}(-1.01,-1.01)(1.01,1.01)
	\SpecialCoor
	\psline[linewidth=2pt](1;180)(.65;225)
	\psline(1;247.5)(.65;225)
	\psdot(1;247.5)
	\psline(0.1;90)(.65;225)
	\psline[linewidth=2pt](0.1;90)(1;112.5)
	\psdot(1;112.5)
	\psline(0.1;90)(0.7;5)
	\psline(1;22.5)(0.7;5)
	\psdot(1;22.5)
	\psline(1;-45)(0.7;5)
	\psdot[dotsize=5pt](1;-45)
	\rput(1.15;180){$p$}
	\rput(.05;0){$c$}
	\rput(1.15;247.5){$x$}
	\rput(1.15;112.5){$y$}
	\rput(.8;0){$z$}
\end{pspicture}}

\newcommand{\figureItreefromngon}{
\psset{unit=0.1667\textwidth, linewidth=0.05pt, dotsize=3.6pt}
\providecommand{\pstrrootedge}{\psline[linestyle=dashed, linewidth=0.3pt, arrows=*-o]}
\providecommand{\pstrintedge}{\psline[linestyle=dashed, linewidth=0.3pt, arrows=*-]}
\providecommand{\pstrextedge}{\psline[linestyle=dashed, linewidth=0.3pt, arrows=o-]}
\providecommand{\addTriangCommand}{}
\begin{pspicture}(-1.01,-1.01)(1.01,1.01)
\addTriangCommand
\psarc[linestyle=dotted,linewidth=0.75pt]{-}(0,0){1}{0}{360}
\pstrintedge(-0.122008,-0.455342)(0.044658,0.166667)
\pstrintedge(0.455342,-0.455342)(-0.122008,-0.455342)
\pstrextedge(-0.965926,-0.258819)(-0.910684,0.000000)
\pstrintedge(-0.577350,0.333333)(0.044658,0.166667)
\pstrintedge(0.455342,0.788675)(0.622008,0.500000)
\pstrextedge(0.258819,0.965926)(0.455342,0.788675)
\pstrrootedge(0.622008,0.500000)(0.965926,0.258819)
\pstrextedge(0.707107,-0.707107)(0.288675,-0.744017)
\pstrintedge(0.288675,-0.744017)(0.455342,-0.455342)
\pstrintedge(-0.000000,-0.910684)(0.288675,-0.744017)
\pstrextedge(0.258819,-0.965926)(-0.000000,-0.910684)
\pstrextedge(0.707107,0.707107)(0.455342,0.788675)
\pstrextedge(-0.707107,0.707107)(-0.455342,0.788675)
\pstrextedge(-0.258819,-0.965926)(-0.000000,-0.910684)
\pstrintedge(-0.910684,0.000000)(-0.577350,0.333333)
\pstrintedge(0.044658,0.166667)(0.622008,0.500000)
\pstrextedge(-0.965926,0.258819)(-0.910684,0.000000)
\pstrintedge(-0.455342,0.788675)(-0.577350,0.333333)
\pstrextedge(-0.258819,0.965926)(-0.455342,0.788675)
\pstrextedge(-0.707107,-0.707107)(-0.122008,-0.455342)
\pstrextedge(0.965926,-0.258819)(0.455342,-0.455342)
\SpecialCoor
\pspolygon(1;210.000000)(1;240.000000)(1;360.000000)
\pspolygon(1;240.000000)(1;330.000000)(1;360.000000)
\pspolygon(1;90.000000)(1;150.000000)(1;210.000000)
\pspolygon(1;30.000000)(1;60.000000)(1;90.000000)
\pspolygon(1;30.000000)(1;90.000000)(1;360.000000)
\pspolygon(1;240.000000)(1;300.000000)(1;330.000000)
\pspolygon(1;240.000000)(1;270.000000)(1;300.000000)
\pspolygon(1;150.000000)(1;180.000000)(1;210.000000)
\pspolygon(1;90.000000)(1;210.000000)(1;360.000000)
\pspolygon(1;90.000000)(1;120.000000)(1;150.000000)
\end{pspicture}}

%% file: Algebraic_two-level_measure_trees_7.bbl
{\footnotesize

\begin{thebibliography}{BRSSJ19}

\bibitem[Ald94a]{Aldous1994a}
David Aldous.
\newblock Recursive self-similarity for random trees, random triangulations and
  {B}rownian excursion.
\newblock {\em Ann. Probab.}, 22(2):527--545, 1994.

\bibitem[Ald94b]{Aldous1994b}
David Aldous.
\newblock Triangulating the circle, at random.
\newblock {\em The American Mathematical Monthly}, 101(3):223--233, 1994.

\bibitem[Ald00]{Aldous2000}
David Aldous.
\newblock Mixing time for a {Markov} chain on cladograms.
\newblock {\em Combinatorics, Probability and Computing}, 9:191--204, 2000.

\bibitem[ALW17]{AthreyaLoehrWinter2017}
Siva Athreya, Wolfgang L{\"o}hr, and Anita Winter.
\newblock Invariance principle for variable speed random walks on trees.
\newblock {\em Ann.\ Probab.}, 45(2):625--667, 2017.

\bibitem[Ban08]{Bansaye2008}
V.~Bansaye.
\newblock Proliferating parasites in dividing cells : Kimmel's branching model
  revisited.
\newblock {\em Annals of Applied Probability}, 2008.

\bibitem[BDLS18]{BlancasBenitezDuchampsLambertSiriJegousse2018}
{Airam Blancas} Ben{\'i}tez, Jean-Jil Duchamps, Amaury Lambert, and Arno
  {Siri-J\'egousse}.
\newblock {Trees within Trees: Simple Nested Coalescents}.
\newblock {\em Electron. J. Probab.}, 23:1--27, 2018.

\bibitem[BJ10]{BlumJakobsson2010}
M.G.B. Blum and M.~Jakobsson.
\newblock {Deep Divergences of Human Gene Trees and Models of Human Origins}.
\newblock {\em Molecular Biology and Evolution}, 28(2):889--898, 10 2010.

\bibitem[BRSSJ19]{BlancasBenitezRogersSchweinsbergSiriJegousse2019}
{Airam Blancas} Ben{\'i}tez, Tim Rogers, Jason Schweinsberg, and Arno
  Siri-J{\'e}gousse.
\newblock {The Nested Kingman Coalescent: Speed of Coming Down from Infinity}.
\newblock {\em Annals of Applied Probability}, 29(3):1808–1836, February
  2019.
\newblock 24 pages.

\bibitem[BT11]{BansayeTran2011}
V.~Bansaye and V.C. Tran.
\newblock Branching {F}eller diffusion for cell division with parasite
  infection.
\newblock {\em ALEA}, 8:95--127, 2011.

\bibitem[Chi01]{Chiswell2001}
I.~Chiswell.
\newblock {\em Introduction to {$\Lambda$}-trees}.
\newblock World Scientific, 2001.

\bibitem[Daw18]{Dawson2018}
D.~Dawson.
\newblock Multilevel mutation-selection systems and set-valued duals.
\newblock {\em J. Math. Biol.}, 76:295--378, 2018.

\bibitem[DGW04]{DawsonGorostizaWakolbinger2004}
Donald Dawson, Luis Gorostiza, and Anton Wakolbinger.
\newblock {Hierarchical Equilibria of Branching Populations}.
\newblock {\em Electronic Journal of Probability}, 9(none):316 -- 381, 2004.

\bibitem[DHV95]{DawsonHochbergVinogradov1995}
Donald~A. Dawson, Kenneth~J. Hochberg, and Vladimir Vinogradov.
\newblock On path properties of super-2 processes ii.
\newblock In M.G. Cranston and M.A. Pinsky, editors, {\em Proceedings of
  Symposia in Pure Mathematics Series}, volume~57, pages 385--403. AMS,
  Providence, 1995.

\bibitem[DHV96]{DawsonHochbergVinogradov1996}
Donald~A. Dawson, Kenneth~J. Hochberg, and Vladimir Vinogradov.
\newblock {High-density limits of hierarchically structured branching-diffusing
  populations}.
\newblock {\em Stochastic Processes and their Applications}, 62(2):191--222,
  July 1996.

\bibitem[DHW90]{DawsonHochbergWu1990}
D.A. Dawson, K.J. Hochberg, and Y.~Wu.
\newblock {\em White Noise Analysis: Mathematics and Applications}, chapter
  Multilevel branching systems.
\newblock World Scientific Publ., 1990.

\bibitem[DMT96]{DressMoultonTerhalle1996}
Andreas Dress, Vincent Moulton, and Werner Terhalle.
\newblock T-theory: An overview.
\newblock {\em European Journal of Combinatorics}, 17(2):161--175, 1996.

\bibitem[Dre84]{Dress1984}
Andreas~W.M. Dress.
\newblock Trees, tight extensions of metric spaces, and the cohomological
  dimension of certain groups: A note on combinatorical properties of metric
  spaces.
\newblock {\em Adv. Math.}, 53:321--402, 1984.

\bibitem[DT96]{DressTerhalle1996}
A.W.M. Dress and W.F. Terhalle.
\newblock The real tree.
\newblock {\em Advances in Mathematics}, 120(2):283--301, 1996.

\bibitem[EK86]{EthierKurtz1986}
Stewart~N. Ethier and Thomas~G. Kurtz.
\newblock {\em Markov {P}rocesses: {C}haracterization and {C}onvergence}.
\newblock John Wiley, New York, 1986.

\bibitem[Eva08]{Evans2008}
Steven~N. Evans.
\newblock {\em Probability and real trees}, volume 1920 of {\em Lecture Notes
  in Mathematics}.
\newblock Springer, Berlin, 2008.
\newblock Lectures from the 35th Summer School on Probability Theory held in
  Saint-Flour, July 6--23, 2005.

\bibitem[For]{Ford2005}
Daniel~J. Ford.
\newblock Probabilities on cladograms: introduction to the alpha model.
\newblock arXiv:math/0511246.

\bibitem[GH00]{GrevenHochberg2000}
Andreas Greven and Kenneth Hochberg.
\newblock New behavioral patterns for two-level branching systems.
\newblock {\em Canadian Mathematical Society Conference Proceedings},
  6:205--215, 01 2000.

\bibitem[GHW95]{GorostizaHochbergWakolbinger1995}
Luis~G. Gorostiza, Kenneth~J. Hochberg, and Anton Wakolbinger.
\newblock Persistence of a critical super-2 process.
\newblock {\em Journal of Applied Probability}, 32(2):534--540, 1995.

\bibitem[GPW09]{GrevenPfaffelhuberWinter2009}
Andreas Greven, Peter Pfaffelhuber, and Anita Winter.
\newblock Convergence in distribution of random metric measure spaces
  ({$\Lambda$}-coalescent measure trees).
\newblock {\em Probab.\ Theo.\ Rel.\ Fields}, 145:285--322, 2009.

\bibitem[Gro99]{Gromov1999}
Misha Gromov.
\newblock {\em Metric structures for {R}iemannian and non-{R}iemannian spaces},
  volume 152 of {\em Progress in Mathematics}.
\newblock Birkh\"auser Boston Inc., Boston, MA, 1999.

\bibitem[JBA19]{JayBoitardAusterlitz2019}
F.~Jay, S.~Boitard, and F.~Austerlitz.
\newblock An {A}{B}{C} method for whole-genome sequence data: inferring
  {P}aleolithic and {N}eolithic human expansions.
\newblock {\em Molecular Biology and Evolution}, 36:1565--1579, 2019.

\bibitem[Kim97]{Kimmel1997}
M.~Kimmel.
\newblock Quasistationarity in a branching model of division-within-division.
\newblock In {\em Classical and modern branching processes (Minneapolis, MN,
  1994)}, volume~84 of {\em IMA Vol. Math. Appl.}, pages 157--164, New-York,
  1997. Springer.

\bibitem[Kin82]{Kingman1982}
J.~F.~C. Kingman.
\newblock The coalescent.
\newblock {\em Stochastic Process. Appl.}, 13(3):235--248, 1982.

\bibitem[LBP+21]{LepersBilliardPorteMeleardTran2021}
C.~Lepers, S.~Billiard, M.~Porte, S.~M\'el\'eard, and V.C. Tran.
\newblock Inference with selection, varying population size and evolving
  population structure: Application of abc to a forward-backward.
\newblock {\em Heredity}, 126:335--350, 2021.

\bibitem[LMW20]{LoehrMytnikWinter2020}
Wolfgang L{\"o}hr, Leonid Mytnik, and Anita Winter.
\newblock The {A}ldous chain on cladograms in the diffusion limit.
\newblock {\em The Annals of Probability}, 48:2565--2590, 09 2020.

\bibitem[L{\"o}h13]{Loehr2013}
Wolfgang L{\"o}hr.
\newblock {Equivalence of Gromov-Prohorov- and Gromov's
  $\underline\Box_\lambda$-metric on the space of metric measure spaces}.
\newblock {\em Electronic Communications in Probability}, 18(none):1 -- 10,
  2013.

\bibitem[LW21]{LoehrWinter2021}
Wolgang L{\"o}hr and Anita Winter.
\newblock Spaces of algebraic measure trees and triangulations of the circle.
\newblock {\em Bulletin de la Société Mathématique de France}, 149(1):1--63,
  2021.

\bibitem[Mei19]{Meizis2019}
Roland Meizis.
\newblock Convergence of metric two-level measure spaces.
\newblock {\em Stochastic Processes and their Applications}, 130, 10 2019.

\bibitem[MNO92]{MayerNikielOversteegen1992}
John~C. Mayer, Jacek Nikiel, and Lex~G. Oversteegen.
\newblock Universal spaces for {R}-trees.
\newblock {\em Transactions of the American Mathematical Society},
  334(1):411--432, 1992.

\bibitem[MO90]{MayerOversteegen1990}
John Mayer and Lex Oversteegen.
\newblock A topological characterization of {R}-trees.
\newblock {\em Trans. Amer. Math. Soc.}, 320:395--415, 07 1990.

\bibitem[MR13]{MeleardRoelly2013}
Sylvie Méléard and Sylvie R\oe{}lly.
\newblock Evolutive two-level population process and large population
  approximations.
\newblock {\em Ann.\ Univ.\ Buchar.\ Math.\ Ser.}, 4(LXII)(1):37--70, 2013.

\bibitem[Nus]{Nussbaumer2021}
Josu\'e Nussbaumer.
\newblock Resampling dynamics on metric two-level measure trees.
\newblock Manuscript in prep.

\bibitem[NW]{NussbaumerWinter2020}
Josu\'e Nussbaumer and Anita Winter.
\newblock The algebraic alpha-ford tree under evolution.
\newblock arXiv:2006.09316.

\bibitem[Ter97]{Terhalle1997}
W.F. Terhalle.
\newblock R-trees and symmetric differences of sets.
\newblock {\em Europ. J. Combinatorics}, 18:825--833, 1997.

\bibitem[Tit77]{Tits1977}
Jacques Tits.
\newblock {A ``theorem of Lie-Kolchin'' for trees}.
\newblock {\em Contributions to Algebra: A Collection of Papers Dedicated to
  Ellis Kolchin}, 12 1977.

\bibitem[VAE+09]{Verduetal2009}
P.~Verdu, F.~Austerlitz, A.~Estoup, R.~Vitalis, M.~Georges, S.~Th\'ery,
  A.~Froment, S.~Le~Bomin, A.~Gessain, J.-M. Hombert, L.~Van~der Veen,
  L.~Quintana-Murci, S.~Bahuchet, and E.~Heyer.
\newblock Origins and genetic diversity of {P}ygmy hunter-gatherers from
  {W}estern {C}entral {A}frica.
\newblock {\em Current Biology}, 19(4):312--318, 2009.

\bibitem[Wu91]{Wu1991}
Yadong Wu.
\newblock {\em Dynamic Particle Systems and Multilevel Measure Branching
  Processes}.
\newblock PhD thesis, Carleton University, Ottawa, Canada, 1991.

\bibitem[Wu94]{Wu1994}
Y.~Wu.
\newblock Asymptotic behavior of the two-level measure branching process.
\newblock {\em Annals of Probab.}, 22(2):854--874, 1994.

\end{thebibliography}
}
